\documentclass[12pt]{article}
\title{Operator limit of the circular beta ensemble}
\date{}
\author{Benedek Valk\'o and B\'alint Vir\'ag}

\usepackage{units}
\usepackage{amsmath, amsthm, amssymb}
\usepackage{graphicx}
\usepackage{amsmath, enumerate}
\usepackage{color}
\usepackage[longnamesfirst]{natbib}
\bibpunct[ ]{(}{)}{,}{a}{}{,}

    \newtheorem{theorem}{Theorem}
    \newtheorem{lemma}[theorem]{Lemma}
    \newtheorem{proposition}[theorem]{Proposition}
    \newtheorem{corollary}[theorem]{Corollary}

\theoremstyle{definition} 

\newcommand{\eps}{\varepsilon}

\newcommand{\ZZ}{{\mathbb Z}}
\newcommand{\FF}{{\mathcal
F}}

\newcommand{\R}{{\mathbb R}}

\newcommand{\HH}{{\mathbb H} }

\newcommand{\lstar}{{\raise-0.15ex\hbox{$\scriptstyle \ast$}}}

\theoremstyle{remark} 
\newcommand{\Sineb}{\operatorname{Sine}_{\beta}}
\newcommand{\Sine}{\operatorname{Sine}}

\newcommand{\Sineop}{\mathtt{Sine}_{\beta}}

\newcommand{\Dirop}{\mathtt{Dir}}
\newcommand{\Circop}{\mathtt{Circ}_{\beta,n}}

\definecolor{violet}{rgb}{0.8,0,0.2}

\newcommand{\ed}{\stackrel{d}{=}}
\newcommand{\cB}{{\mathcal B}}

\newcommand{\mat}[4]{\left( \begin{array}{cc}
#1 & #2  \\
#3 & #4  \\
\end{array} \right)}

\newcommand{\ac}{{\text{\sc ac}}}

\newcommand{\tr}{\operatorname{tr}}
\newcommand{\ind}{\mathbf 1}

\newcommand{\benedek}[1]{\textcolor{red}{\texttt{#1}}}

\newcommand{\sech}{{\operatorname{sech}}}

\newcommand{\res}{{\mathtt{r}\,}}

\newcommand{\ress}{{\mathtt{r}}}

\newcommand{\HS}{\textup{HS}}

\newcommand{\bside}{\noindent\textbf{\benedek{Begin side computation.}}
\begin{footnotesize}}

\newcommand{\eside}{\end{footnotesize}
\noindent \textbf{\benedek{End side computation.}}}
\newcommand{\spec}{\operatorname{spec}}

\begin{document}
\maketitle

%
\begin{abstract}
We provide a precise coupling of the finite circular beta ensembles and their limit process via their operator representations. We prove explicit bounds on the distance of the operators and the corresponding point processes. We also prove an estimate on the beta-dependence in the $\Sineb$ process. 
\end{abstract}

\begin{figure*}[ht]\hskip20mm
\includegraphics*[width=330pt]{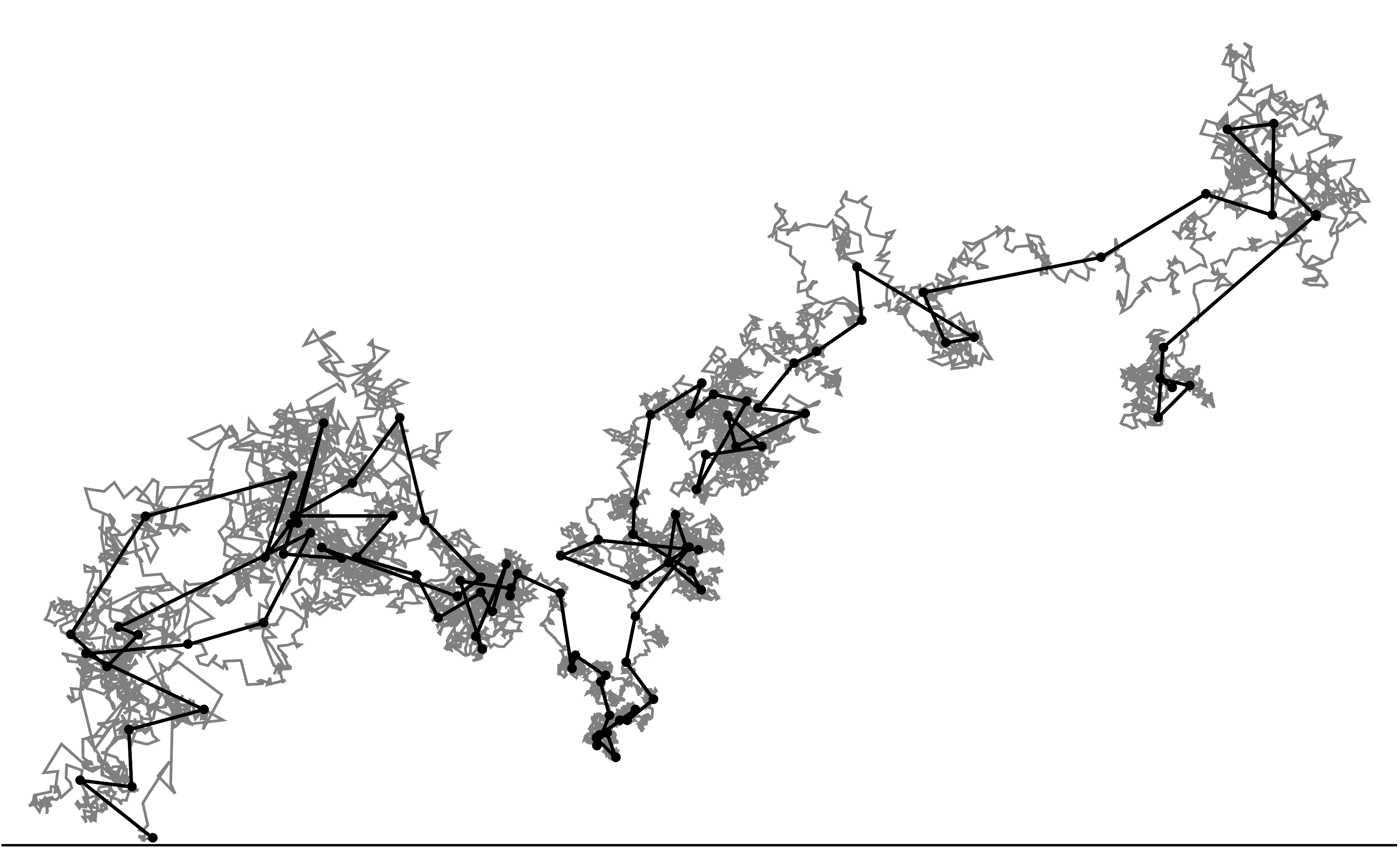}\\
\caption{Simulation of  hyperbolic Brownian motion and a coupled random walk.}
\end{figure*}

\section{Introduction}

\cite{BVBV_sbo}  introduced a family of differential operators parametrized by a path
$\gamma:[0,1)\to \HH$ in the upper half plane and two points on the boundary.

When the path
is a certain hyperbolic random walk in the Poincar\'e half-plane model, the operator $\Circop$ has eigenvalues given by the points of the  circular beta ensemble  scaled and lifted periodically to the real line. With the path  $\gamma(t)=\cB(-\frac{4}{\beta}\log(1-t))$ where $\cB$ is standard hyperbolic Brownian motion,  the operator $\Sineop$ has eigenvalues given by the $\Sineb$ process, the limit of the circular beta ensembles. (See Theorems \ref{thm:sbo} and \ref{thm:circop}.)

The inverses of these operators in a compatible basis are integral operators denoted by $\res \Circop$ and $\res \Sineop$, respectively. Our main result is a coupling which gives $\Circop \to \Sineop$ with  an explicit rate of convergence. (See Section \ref{s:op} for additional details.)

\begin{theorem}\label{thm:main}
There exist an array of stopping times $0=\tau_{n,n}<\tau_{n,n-1}<\dots<\tau_{n,0}=\infty$  so that $\cB(\tau_{n,\lceil (1-t)n \rceil}), t\in [0,1)$ has the law of the random walk on $\HH$
used to generate $\Circop$.

This coupling of $\Sineop$ and  the sequence of operators $\Circop$ satisfies
\begin{equation}\label{mainbound}
\|\res \Sineop - \res \Circop\|^2_{\HS} \le  \frac{\log^{6} n}{n}
\end{equation}
a.s. in the Hilbert-Schmidt norm for all $n\ge N$, where $N$ is random.
\end{theorem}

As a corollary, we get new results about the rate of convergence of the eigenvalue processes.
Let $\lambda_k, k\in \ZZ$ be the ordered sequence of eigenvalues of $\Sineop$ with $\lambda_0<0<\lambda_1$, the sequence $\lambda_{k,n}, k\in \ZZ$ is defined analogously for $\Circop$.
\begin{corollary}\label{cor:cpl} In the coupling of Theorem \ref{thm:main} for all $n\ge N$ we have a.s.
\begin{equation}\label{HSev}
\sum_{k\in \ZZ}\left(\lambda^{-1}_k-\lambda^{-1}_{n,k}\right)^2 \le  \frac{\log^{6} n}{n}.
\end{equation}
Moreover, as $n\to \infty$ we have a.s.
\begin{align}\label{cpl}
\max_{|k|\le \frac{n^{1/4}}{\log^{2} n}} |\lambda_k-\lambda_{k,n}|\to 0.
\end{align}
For all $\eps>0$ there is a random $N_\eps$ so that for $n\ge N_\eps$ and  $|k|\le n^{1/2-\eps}$ we have 
\begin{equation}\label{joseph}
|\lambda_k-\lambda_{k,n}|\le \frac{1+k^2}{n^{1/2-\eps}}.
\end{equation}
\end{corollary}
This provides the best known coupling of the circular $\beta$-ensemble to the $\Sineb$ process, even for $\beta=2$, when both processes are determinantal with explicitly given kernels.
For $\beta=2$ the bound \eqref{joseph} improves on a coupling given in \cite{MNN} in which the inequality
holds with the exponent $1/3$ instead of $1/2$, for $|k|\le n^{1/4}$.

Using the techniques introduced in our proof we are also able to give an estimate on the dependence on $\beta$ in the $\Sineb$ process.  
\begin{theorem}\label{thm:betadep}
Construct the $\Sineop$ operators for all $\beta>0$ with  the same hyperbolic Brownian motion. Denote the eigenvalues corresponding to $\beta$ by  $\{\lambda_{k,\beta}, k\in \ZZ\}$ with $\lambda_{0,\beta}<0\le \lambda_{1,\beta}$. Then for  $0<\kappa<1$ there is an a.s.~finite $C=C_\kappa$ so that if $\kappa<\beta_1<\beta<\kappa^{-1}$ and 
 $\delta=\left|\frac4{\beta}-\frac4{\beta_1}\right|\le 1/3$ then 
\begin{align}\label{sb1}
\sum_{k} \left(\frac1{\lambda_{k,\beta}}-\frac1{\lambda_{k,\beta_1}}\right)^2\le \|\res \Sineop-\res \mathtt{Sine}_{\beta_1}   \|_{\HS}   \le C\delta \log(\tfrac1{\delta}).
\end{align}
\end{theorem}

\subsection*{Sructure of the paper}
The proof relies on a precise coupling of a hyperbolic random walk and hyperbolic Brownian motion. 

The starting point is a hyperbolic heat kernel bound. Consider
the squared Euclidean norm for a hyperbolic Brownian motion  started at the origin 
in the the Poincar\'e disk model. We show that this quantity is very close in total variation to a beta random variable, which also stochastically dominates it (Lemma \ref{lem:betacpl} in Section \ref{s:1cpl}). 

This leads to a coupling of the hyperbolic random walk steps
with hyperbolic Brownian increments at stopping times (Sections \ref{s:1step} and \ref{s:path}). For  each single step, with high probability we stop at a fixed time. Otherwise, we just wait for the Brownian motion to 
hit the right distance. At the tail of the random walk a slightly different coupling is implemented.

A modulus of continuity estimate for hyperbolic Brownian
motion then implies that the random walk is close to the Brownian path (Section \ref{s:pathcmp}).

In Section \ref{s:HS} we show that if two paths are close and escape linearly to the boundary of $\HH$ then 
the corresponding operators are also close. Finally, in Section \ref{s:proof} we use the the coupling and the linear rate of escape for hyperbolic Brownian motion to prove Theorem \ref{thm:main}. 

Section \ref{s:beta} proves Theorem \ref{thm:betadep}.
Some of the technical facts needed are collected in the Appendix. 

\subsection*{Historical background}

The modern history of random matrices originates from \cite{Wigner51}, who used them to approximate the spectrum of 
self-adjoint operators from statistical physics point of view. 

In the following decades the scaling behavior of a number of random matrix models were derived.  The point process limits of the random matrix spectra were described via the limiting joint densities, usually relying on some algebraic structure of the finite models. (See the monographs \cite{AGZ}, \cite{ForBook}, \cite{mehta} for an overview of the classical  results.)

\cite{DE} constructed tridiagonal random matrix models with spectrum distributed as  beta ensembles,  one parameter extensions of classical random matrix models. \cite{ES} observed that under the appropriate scaling these tridiagonal matrix models behave like approximate versions of random stochastic operators, and conjectured that scaling limits of beta ensembles can be described as the spectra of these objects . 

These conjectures were confirmed in \cite{RRV} and \cite{RR} for the the soft and hard edge scaling limits of beta ensembles. The authors rigorously defined the stochastic differential operators that show up as limits, and proved the convergence of the finite ensembles to the spectrum of these operators. 

In \cite{BVBV} and \cite{KS} the bulk scaling limit of the gaussian and circular beta ensembles were derived, and  the counting functions of the limit processes were characterized via coupled systems of SDEs. In \cite{Nakano} and \cite{BVBV_sbo} it was shown that the scaling limit of the circular beta ensemble is the same as $\Sineb$, the bulk limit of the gaussian beta ensemble. Furthermore, \cite{BVBV_sbo} constructed a stochastic differential operator with a spectrum given by $\Sineb$. 

The coupling of the circular beta ensemble for $\beta=2$ (the circular unitary ensemble) to its limit, the $\Sine_2$ process has been recently studied in \cite{MNN} and \cite{MM}. In \cite{MNN} the  circular unitary ensembles of various sizes are coupled together and it is shown that the scaled ensembles converge a.s.~to a $\Sine_2$ process. Moreover, a bound of the form (\ref{joseph}) is given with an exponent $1/3$ instead of $1/2$. 
In \cite{MM} the total  variation distance between the counting functions of the finite and the limiting process is considered. Denote by $\mathcal{N}_n$  the counting function of the appropriately scaled circular unitary ensemble of size $n$, and  by $\mathcal{N}$ the counting function of the $\Sine_2$ process. It is shown that for any fixed interval $I$  the following bound holds:
\begin{align}\label{MM_TV}
d_{TV}(\mathcal{N}_n(I), \mathcal{N}(I))\le 5 \frac{|I|^2}{n^{3/2}}, \qquad \text{for $n\ge n_0(I)$.} 
\end{align} 
This provides a bound on the distance between the distributions for the number of points in a given interval, but does not seem to imply a process level result.

\section{Stochastic differential operators}\label{s:op}

We review the framework introduced in \cite{BVBV_sbo} to study random matrix ensembles via differential operators.

%
%


\subsection{Dirac operators}

We consider differential operators of the form
\begin{align}\label{ttt}
\tau: f\to R^{-1}(t) J \frac{d}{dt}f.
\end{align}
Here $f:[0,1)\to \R^2$, and
\begin{align}\label{RXX}
J=\mat{0}{-1}{1}{0}, \qquad R=\frac12 X^tX, \qquad X=\frac{1}{\sqrt{y}}\mat{1}{-x}{0}{y},
\end{align}
with $x:[0,1)\to \R$ and $y:[0,1)\to (0,\infty)$. We consider boundary conditions parametrized by  nonzero vectors $u_0, u_1\in \R^2$ where we assume that $u_0^tJu_1=1$.
We set the domain of the differential operator $\tau$ as
\begin{align}\label{domLC}
\textup{dom}(\tau)=\{v\in L^2_R \cap  \ac\,:\,  \, \tau v\in L^2_R,\,v(0)^t\,J\,u_0=0, \,\lim_{s\to 1} v(s)^t\,J\,u_1 = 0\}.
\end{align}
Here $L^2_R$ is the $L^2$ space of functions $f:[0,1)\to \R^2$ with the $L^2$ norm $\|f\|_2^2=\int_0^1 f^t R f ds$ while $\ac$ is the set of absolutely continuous functions.

The function $\gamma=x+i y$ is a path in the upper half plane $\{(x,y): y>0\}$. In \cite{BVBV_sbo} it was shown that various  properties of $\tau$ can be identified by treating $\gamma$ as a path in the  hyperbolic plane $\HH$ (using the upper half plane representation) with $u_0, u_1$ identified with boundary points $\eta_0, \eta_1$ of $\HH$. The set of boundary points of $\HH$ in the upper half plane representation is $\R\cup \{\infty\}$.  A nonzero $u=(u_1,u_2)^t\in \R^2$ corresponds to the boundary point $\mathcal P u=\frac{u_1}{u_2}$ where this is $\infty$ if $u_2=0$.
To show the dependence on these parameters, we use the notation $\tau=\Dirop(\gamma,\eta_0, \eta_1)$. 

%
%
%
%

For a given boundary point $\eta\in \partial \HH$ the (signed) horocyclic distance of points $a$ and $b$ in $\HH$ with respect to $\eta$ is defined as
\begin{align*}
d_\eta(a,b)=\lim_{z\to \eta}\left(d_{\HH}(a,z)-d_{\HH}(b,z)\right).
\end{align*}
Here $d_{\HH}$ is the hyperbolic distance and the limit is evaluated along a sequence of points in $\HH$ converging to $\eta$. We record the following formulas for the half-plane representation:
\begin{align}\label{horo}
d_\eta(x+i  y, i)=\begin{cases}
\log(\tfrac1{y}), &\qquad \text{if }\eta=\infty,\\[5pt]
\log\left(\frac{(x-q)^2+y^2}{(1+q^2)y}\right),&\qquad \text{if }\eta=q\in \R.
\end{cases}
\end{align}

The following theorem gives a condition in terms of the parameters $\gamma, \eta_0, \eta_1$ for  $\tau$ to be self-adjoint with a Hilbert-Schmidt inverse.
\begin{theorem}[\cite{BVBV_sbo}]\label{thm:Dirop}
 Let $\eta_0, \eta_1$ be distinct boundary points of $\HH$ and   $\gamma:[0,1)\to \HH$ be measurable and locally bounded.  Assume that
there is a $\xi\in \HH$ with
\begin{align}\label{horo_int1}
\int_0^T  e^{d_{\eta_1}(\gamma(t),\xi)}dt<\infty, \quad \text{and} \quad
\int_0^T \int_s^T e^{d_{\eta_0}(\gamma(s),\xi)+d_{\eta_1}(\gamma(s),\xi)} ds dt<\infty.
\end{align}
Then the operator $\tau=\Dirop(\gamma,\eta_0, \eta_1)$ is self-adjoint on $\textup{dom}(\tau)$ and its inverse is Hilbert-Schmidt. The inverse $\tau^{-1}$ is an integral operator on $L^2_R$ with kernel function
\begin{align}\label{HSkernel1}
K(x,y)=\left(u_0 u_1^t  \ind(x<y)+u_1 u_0^t \ind(x\ge y)  \right) R(y).
\end{align}
\end{theorem}
Suppose that $\gamma,\eta_0, \eta_1$ satisfy the conditions of the theorem above and consider the operator $\tau=\Dirop(\gamma,\eta_0, \eta_1)$. Let $\hat \tau=X\tau  X^{-1}$, this is just $\tau$ after a  change of coordinates. In particular $\hat \tau$  is a self-adjoint differential operator on $\{X v\in \textup{dom}(\tau)\}\subset L^2$, with the same spectrum as $\tau$. We denote the inverse of $\hat \tau$ by $\res \tau$ ($\res$ standing for resolvent).
By Theorem \ref{thm:Dirop} the operator $\res \tau$
 is an integral operator acting on $L^2$ functions with kernel
\begin{align}\label{HSkernel}
K_{\res \tau}(x,y)=\frac12 \left(a(x) c(y)^t  \ind(x<y)+ c(x) a(y)^t \ind(x\ge y) \right)
\end{align}
where $a(s)=X(s)u_0$ and $c(s)= X(s) u_1$.

%

\subsection{Stochastic operators}

The gaussian and circular $\beta$ ensembles are defined via the following joint densities on $\R^n$ and $[0,2\pi)^n$:
\begin{align}\label{gbetadens}
p_{\beta,n}^{\text{g}}(\lambda_1,\dots,\lambda_n)&=\frac{1}{Z_{n,\beta}^{\text{g}}} \prod_{1\le j<k\le n} |\lambda_j-\lambda_k|^\beta e^{-\tfrac{\beta}{4}\sum_{j=1}^n \lambda_j^2},
\\p_{\beta,n}^{\text{c}}(\lambda_1,\dots,\lambda_n)&=\frac{1}{Z_{n,\beta}^{\text{c}}} \prod_{1\le j<k\le n} |e^{i \lambda_j}-e^{i \lambda_k}|^\beta. \label{cbetadens}
\end{align}
For $\beta=2$ these give the joint eigenvalue densities of the gaussian and circular unitary ensemble, respectively. We use angles to represent the eigenvalues in the circular case.

The bulk scaling limits of these ensembles have been identified in  \cite{BVBV} and \cite{KS}. In \cite{Nakano} and \cite{BVBV_sbo} it was shown that the scaling limit of the circular beta ensemble is the same as the bulk limit of the gaussian beta ensemble.
\begin{theorem}[\cite{BVBV}]\label{thm:car}
Fix $\beta>0$ and $|E|<2$. Let $\Lambda_n^{\textup{g}}$ be a finite point process with density (\ref{gbetadens}). Then $\sqrt{4-E^2} \sqrt{n} (\Lambda_n^{\text{g}}-\sqrt{n} E)$ converges in distribution to a point process $\Sineb$.
\end{theorem}
\begin{theorem}[\cite{KS}, \cite{Nakano}, \cite{BVBV_sbo}]\label{thm:cbeta}
Fix $\beta>0$ and  let $\Lambda_n^{\textup{c}}$ be a finite point process with density (\ref{cbetadens}). Then $n \Lambda_n^{\text{c}}$ converges in distribution to the point process $\Sineb$.
\end{theorem}

In \cite{BVBV_sbo} the authors constructed random Dirac operators with spectrum given by $\Sineb$ and the finite circular beta ensemble. Recall that the standard hyperbolic Brownian motion $x+i y$ in the upper half plane representation started from $i$ is the solution of the SDE
\[
d(x+i y)=y(dB_1+i dB_2), \qquad x(0)+i y(0)=i,
\]
where $B_1, B_2$ are independent standard real Brownian motions.

\begin{theorem}[\cite{BVBV_sbo}]\label{thm:sbo}
Fix $\beta>0$ and let $\cB$ be a  standard  hyperbolic Brownian motion in the upper half plane  started from $\cB(0)=i$. Set $\tilde \cB(t)=\cB(-\frac{4}{\beta}\log(1-t)), t\in [0,1)$, $\eta_0=\infty$ and $\eta_1=\lim_{t\to \infty} \cB(t)$. Then the operator
\[
\Sineop=\Dirop(\tilde \cB(t),  \eta_0, \eta_1)
\]
is a.s.~self-adjoint with a Hilbert-Schmidt inverse and $\spec(\Sineop)\ed\Sineb$.
\end{theorem}

The random Dirac operator for the finite circular beta ensemble is defined using a random walk in $\HH$.
Fix $n\ge 1$ and let $\zeta_0, \dots, \zeta_{n-2}$ be independent with $\zeta_k\sim$ {Beta}$(1,\tfrac{\beta}{2}(n-k-1))$. \footnote{
A random variable has distribution Beta$(a,b)$ with $a,b>0$ if it has density $\frac{x^{a-1}(1-x)^{b-1}}{\Gamma(a)\Gamma(b)} \ind(x\in [0,1])$.}
Set $Y_k=\log(\frac{1+\sqrt{\zeta_k}}{1-\sqrt{\zeta_k}})$. We define the random walk $b_0, b_1, \dots, b_{n-1}$ in $\HH$ with a final boundary point $b_{n}\in \partial \HH$. We set $b_0=i$, and for $0\le k\le n-2$ we choose $b_{k+1}$ uniformly among the points in $\HH$ with hyperbolic distance $Y_k$ from $b_k$, independently of the previous choices. The final point $b_n$ is chosen uniformly on the boundary $\partial \HH$ as viewed from $b_{n-1}$, independently of the previous choices. Note that $\zeta_k\sim $ {Beta}$(1,\tfrac{\beta}{2}(n-k-1))$ is the squared Euclidean norm of the random walk step in the Poincare disk model with $b_k$ at the origin.
%
%
%
%
%
%

%
%
\begin{theorem}[\cite{BVBV_sbo}]\label{thm:circop}
Consider the random walk $b_0, b_1, \dots, b_n$ defined above. Set $\eta_0=\infty$, $\eta_1=b_n$ and $\cB_n(t)=b_{\lfloor n t\rfloor}$ for $t\in[0,1)$. The operator
\[
\Circop=\Dirop(\cB_n(t),  \eta_0, \eta_1)
\]
is a.s.~self-adjoint with a Hilbert-Schmidt inverse and $\spec(\Circop)\ed n \Lambda_n^{\textup{c}}+2\pi \ZZ$ where $\Lambda_n^{\textup{c}}$ is the finite point process with density (\ref{cbetadens}).
\end{theorem}

Theorems \ref{thm:sbo} and \ref{thm:circop} imply that almost surely zero is not an eigenvalue for  the operators $\Sineop$ and $\Circop$.

\subsection{Coupling $\Circop$ and $\Sineop$ }

Fix $\beta>0$. Let $\cB$ a hyperbolic Brownian motion, set $\eta_0=\infty$, $\eta_1=\cB(\infty)$ and  $\Sineop=\Dirop(\cB(-\frac{4}{\beta}\log(1-t)),\eta_0,\eta_1)$.

Let $U$ be a uniform random variable on $[0,1]$ independent from $\cB$, and let $\FF_t, t\ge 0$ be the natural filtration of $\cB$ enlarged with $U$. Theorem \ref{thm:main} gives the existence of an array of stopping times $\tau_{n,k}$ with respect to $\FF_t$ so that for each $n\ge 1$ the random variables $\cB(\tau_{n,n}),\cB(\tau_{n,n-1}), \dots, \cB(\tau_{n,0})$ have the same joint distribution  as the random walk $b_0, \dots, b_n$ from Theorem \ref{thm:circop}. Setting  $\Circop=\Dirop(\tilde \cB_n(t), \eta_0, \eta_1)$  with $\tilde \cB_n(t)=\cB(\tau_{n,\lceil (1-t)n \rceil}), t\in [0,1)$ gives the appropriate coupling of $\Sineop$ and the sequence $\{\Circop\}_{ n\ge 1}$.

\section{Coupling construction}

\subsection{A heat kernel bound on the hyperbolic plane}\label{s:1cpl}

Our coupling relies on a careful estimate of the transition density for hyperbolic Brownian motion. Although there are a number of existing  bounds in the literature (see e.g.~\cite{DM}), we could not find one that would be strong enough for our purposes.
We show that the distribution of the distance of hyperbolic Brownian motion at time $t\le 1$  from its starting point  can be very precisely estimated in terms of a \textup{Beta}$(1,\tfrac2{t}-1/2)$ random variable.

\begin{lemma}\label{lem:betacpl}
Let $\cB(t)$ be standard hyperbolic Brownian motion and let $t\in(0,1]$. Let $Y=\log(\frac{1+\sqrt{\xi}}{1-\sqrt{\xi}})$ where $\xi$ has distribution \textup{Beta}$(1,\tfrac2{t}-1/2)$ and set $\zeta=d_\HH(\cB(0),\cB(t))$.
Then the following statements hold:
\begin{enumerate}[(a)]
\item \label{lem:dom} $P(Y>r)\ge P(\zeta>r) $ for all $r>0$, in other words $Y$ stochastically dominates $\zeta$.
\item \label{lem:TV} The total variation distance of $\zeta$ and $Y$ is bounded by $\tfrac32 t$.
\end{enumerate}
\end{lemma}
The proof of the lemma relies on a precise analysis of the explicit formula for the transition density. 
We leave it for Section \ref{s:calc}    in the Appendix.

\subsection{Single step coupling}\label{s:1step}

We first concentrate on a single step in the hyperbolic random walk corresponding  to $\Circop$ and couple it to the hyperbolic Brownian motion.

\begin{proposition}\label{prop:1step}
Fix $\gamma \ge 3/2$. Let $\cB$ be hyperbolic Brownian motion and $U$ an independent uniform random variable on $[0,1]$.  Let $\FF_t, t\ge 0$ be the filtration of $\cB$ enlarged with $U$.  Consider a Poincar\'e disk representation of the hyperbolic plane where $\cB(0)=0$.

There exists a finite random variable $\sigma>0$ so that the following hold:
\begin{enumerate}
\item $\sigma$ is a stopping time with respect to $\FF_t$.
\item $\cB(\sigma)$ has rotationally invariant distribution and $|\cB(\sigma)|^2$ has \textup{Beta}$(1,\gamma)$ distribution.
\item  $P(\sigma\ge \frac{4}{2 \gamma +1})=1$ and $P(\sigma\neq \frac{4}{2 \gamma +1}) \le \tfrac{3}{\gamma}$.
\item For $r\ge 8$ we have $P(\sigma > r/\gamma) \le 3e^{-\frac{1}{5}r^{1/3}}$. 
\end{enumerate}
\end{proposition}
The proof of the proposition will rely on Lemma \ref{lem:betacpl} and the lemma below,  to be proved in the Appendix.
%
%
%

\begin{lemma}\label{lem:stdom1}
Assume that $X_1$ and $X_2$ are random variables so that $X_2$ stochastically dominates $X_1$  and the  total variation distance of their distributions is $\eps$. Then there exists a measurable function $g:\R^2\to \R$ so that if $U$ is a uniform random variable on $[0,1]$,  independent of $X_1$ then the following hold:
\begin{enumerate}[(a)]
\item $g(X_1,U)$ has the same distribution as $X_2$.

\item $P(X_1\le g(X_1,U))=1.$

\item $P(X_1=g(X_1,U))=1-\eps$.

\end{enumerate}
\end{lemma}

\begin{proof}[Proof of Proposition \ref{prop:1step}]
Set $t=\frac{4}{2\gamma+1}$, then $0<t<1$. Recall that if $z$ is in the Poincar\'e disk with  $|z|=r<1$ then $d_{\HH}(0,z)=\log\left(\frac{1+r}{1-r}\right)$.

Let $\xi$ be a random variable with distribution Beta$(1,\gamma)$.  Then by Lemma \ref{lem:betacpl}  the random variable $\log\left(\frac{1+|\cB(t)|}{1-|\cB(t)|}\right)$ is stochastically dominated by $\log(\frac{1+\sqrt{\xi}}{1-\sqrt{\xi}})$  and their total variation distance is bounded by $\tfrac32 t$. Since $\log\left(\frac{1+r}{1-r}\right)$ is strictly increasing in $r$, we get that
 $|\cB(t)|$ is stochastically dominated by $\sqrt{\xi}$ and their total variation distance is bounded by $\tfrac32 t$.

By Lemma \ref{lem:stdom1} there exits a measurable function $g$ so that  $g(|\cB(t)|,U)\ge |\cB(t)|$ almost surely, $g(|\cB(t)|,U)$ has the same distribution as $\sqrt{\xi}$, and $P(|\cB(t)|\neq g(|\cB(t)|,U))\le \tfrac32 t\le 3/\gamma$.

We set
\[
\sigma=\inf\{s\ge t: |\cB(s)|=g(|\cB(t)|,U)\}.
\]
Then $\sigma$ is an a.s.~finite stopping time with respect to $\FF_t$ and almost surely $\sigma\ge t$. Because $\sigma$ only depends on $|\cB|$ and $U$, it follows that $\cB(\sigma)$ has rotationally invariant distribution. Finally, from our construction we get that $|\cB(\sigma)|^2=g(|\cB(t)|,U)^2$  has Beta$(1,\gamma)$ distribution and $P(\sigma \neq t)\le \tfrac32 t\le 3/\gamma$.

The only thing left  to prove is the tail bound for $\sigma$. We start with the  bound
\begin{align}\label{tautail}
P(\sigma>r/\gamma) \le P\left(\sigma>r/\gamma\text{ and }d_\HH(0,\cB(\sigma))\le  \frac{r^{1/3}}{\gamma^{1/2}}\right) + P\left(d_\HH(0,\cB(\sigma))> \frac{r^{1/3}}{\gamma^{1/2}}\right).
\end{align}

For the rest of the proof we assume $r\ge 8$. Then $\sigma>r/\gamma>t$ and from the definition of $\sigma$ it follows that $d_{\HH}(0,\cB(s))<d_{\HH}(0,\cB(\sigma))$ for $t\le s<\sigma$.
Thus we can bound the first term  on the right of  (\ref{tautail}) by writing
\begin{align*}
P\left(\sigma>r/\gamma\text{ and }d_\HH(0,\cB(\sigma))\le  \frac{r^{1/3}}{\gamma^{1/2}}\right)&\le P\left(\max_{t\le s \le r/\gamma}d_\HH(0,\cB(s))\le  \frac{r^{1/3}}{\gamma^{1/2}}\right),\\
&\le  P\left(\max_{t\le s \le r/\gamma}d_\HH(\cB(t),\cB(s))\le  2\frac{r^{1/3}}{\gamma^{1/2}}\right)\\
&= P\left(\max_{0\le s \le r/\gamma-t}d_\HH(0,\cB(s))\le 2 \frac{r^{1/3}}{\gamma^{1/2}}\right).
\end{align*}
Since  $t\le \frac{3}{2\gamma}$ and $r\ge 8$, we have $r/\gamma-t\ge \tfrac45 r \gamma^{-1}$.
Using the bound (\ref{lowerhBM}) from Lemma \ref{lem:hypBMtau} of the Appendix we get
\begin{align*}
P(\max_{0\le s \le r/\gamma-t}d_\HH(0,\cB(s))\le 2 \frac{r^{1/3}}{\gamma^{1/2}})
\le P(\max_{0\le s \le \frac45 r \gamma^{-1}}d_\HH(0,\cB(s))\le  2\frac{r^{1/3}}{\gamma^{1/2}})\le \frac{4}{\pi}e^{-\frac{\pi^2 r^{1/3}}{40 }}.
\end{align*}
%

For the second term in  (\ref{tautail}) we recall that by construction $d_\HH(0,\cB(\sigma))$ has the same distribution as
$\log(\frac{1+\sqrt{\xi}}{1-\sqrt{\xi}})$ where $\xi$ has distribution Beta$(1,\gamma)$. By an explicit computation
\begin{align*}
P\left(\log(\tfrac{1+\sqrt{\xi}}{1-\sqrt{\xi}})> u\right)=\sech^{2\gamma}(\tfrac{u}{2}).
\end{align*}
We have $\sech(x)\le 2 e^{-x}$ which gives the following upper bound for $u\ge 4 \log 2$:
\[
P\left(\log(\tfrac{1+\sqrt{\xi}}{1-\sqrt{\xi}})> u\right)\le 2^{2\gamma} e^{-\gamma u}\le e^{-\frac{\gamma u}{2}}.
\]
For $0\le u\le 4 \log 2$ we have $\log \sech(u/2)\le -\frac{u^2}{12}$ so for these values we get
\[
P\left(\log(\tfrac{1+\sqrt{\xi}}{1-\sqrt{\xi}})> u\right)\le e^{-\frac{\gamma u^2}{6}}.
\]
From this we get
\[
P(d_\HH(0,\cB(\sigma))> \frac{r^{1/3}}{\gamma^{1/2}})\le \max\left( e^{-\frac{\sqrt{\gamma}  r^{1/3}}{2}}, e^{-\frac{ r^{2/3}}{6}}  \right)\le e^{-\tfrac13 r^{1/3}},
\]
where in the last step we used $\gamma\ge 3/2$ and $r\ge 8$.
Collecting our estimates we get
\[
P(\sigma>r/\gamma) \le\frac{4}{\pi}e^{-\frac{\pi^2 r^{1/3}}{40 }}+ e^{-\tfrac13 r^{1/3}}\le 3 e^{-\tfrac15 r^{1/3}},
\]
which completes the proof of the proposition.
\end{proof}

\subsection{Path coupling}\label{s:path}


\begin{proposition}\label{prop:stoptimes}
Let $\cB$ be hyperbolic Brownian motion and $U$ an independent \textup{Uniform}$[0,1]$ random variable. Let $\FF_t, t\ge 0$ be the filtration of $\cB$ enlarged with $U$.

There exists a collection of stopping times  $(\tau_{n,k}; 1\le n, 0\le k\le n)$ with respect to $\FF_t, t\ge 0$  so that the following statements hold.
\begin{enumerate}
\item For each fixed $n$ we have $0=\tau_{n,n}<\tau_{n,n-1}\ldots <\tau_{n,0}=\infty$, and $\Delta \tau_{n,k}:=\tau_{n,k}-\tau_{n,k+1}$ are independent.
\item For each $n$, the process $(b_{k}^{(n)}=\cB(\tau_{n,n-k}), k =0,1,\ldots, n)$ is a hyperbolic random walk with the same distribution as the one given before Theorem \ref{thm:circop}.

\item Let $t_{n,k}=\frac{4}{\beta}\log\left(\frac{n}{k}\right)$ for $1\le k\le n$.
There exists a  random integer $N_0>0$ so that for $n\ge N_0$ and $\log^6 n\le k\le n$ we have almost surely
\begin{align}\label{pathtime}
t_{n,k} - \frac{4 }{\beta^{2}k} \le \tau_{n,k}\le t_{n,k} + \frac{\log^{4+1/2} n}{k}.
\end{align}

\item For $k$ fixed the hyperbolic distance $d_\HH(\cB(\tau_{n,k}),\cB(\tau_{n,k-1}))$ does not depend on $n$ as long as
$k<\log^6 n$.
\end{enumerate}
\end{proposition}

\begin{proof}
Using a standard argument we can find measurable functions $f_{1,k}, f_{2,k}$ so that $U_k=f_{1,k}(U)$ is a Uniform$[0,1]$ random variable, $\xi_k=f_{2,k}(U)$ is distributed as  Beta$(1,\frac{\beta}{2}k)$ and $U_1, U_2, \dots, \xi_1, \xi_2, \dots$ are independent.

We first give the construction of the stopping times, then prove that they satisfy all the conditions.

For a fixed $n\ge 1$ we define $\tau_{n,k}$ recursively, starting with $\tau_{n,n}=0$. If for a certain $k<n$ we have already defined $\tau_{n,k+1}$ then we define $\tau_{n,k}$ as follows.
\begin{itemize}
\item \textbf{In the case of  $k\ge \max(  \log^6 n ,\frac{3}{\beta})$:}\\ We apply Proposition \ref{prop:1step} with $\gamma=\frac{\beta}{2}k$
 for the hyperbolic Brownian motion $\tilde \cB(t)=\cB(t+\tau_{n,k+1}), t\ge 0$ and the independent uniform random variable $U_k$, and denote the constructed stopping time by $\sigma_{n,k}$. We set $\tau_{n,k}=\tau_{n,k+1}+\sigma_{n,k}$.

\item \textbf{In the case of  $1\le k<\max( \log^6 n,\frac{3}{\beta})$:}\\ We set
\[
\tau_{n,k}=\inf\left\{t\ge \tau_{n,k+1}: d_\HH(\cB(\tau_{n,k+1}), \cB(t))=\log\left(\tfrac{1+\sqrt{\xi_k}}{1-\sqrt{\xi_k}}\right)  \right \}.
\]

\item For $k=0$ we define $\tau_{n,k}=\infty$.

\end{itemize}

By construction, the random variables $\tau_{n,k}$ are a.s.~finite stopping times with respect to the filtration $\FF_t, t\ge 0$, and they satisfy conditions 1, 2 and 4.
To check Condition 3 we first choose $n_0$ so  that for $n\ge n_0$ we have $\log^6 n\ge \frac{3}{\beta}$ and $\log^{3+3/8} n\ge 8$. During the rest of this proof we will assume $1\le \log^6 n\le k\le n$.

For $n\ge n_0$,  by Proposition \ref{prop:1step} we have almost surely
\[
\tau_{n,k}\ge \sum_{j=k}^{n-1} \frac{4}{\beta j+1} \ge \frac{4}{\beta}\log\left(\frac{n}{k}\right)-\frac{4}{\beta^2} \frac1k.
\]
This takes care of the lower bound in (\ref{pathtime}).

For the upper bound recall the definition of $\sigma_{n,k}=\tau_{n,k}-\tau_{n,k+1}$. From Proposition \ref{prop:1step} we have the following estimates:
\begin{align}\notag
&P\left(\sigma_{n,k}\ge \tfrac{4}{\beta k +1}\right)=1, \quad P\left(\sigma_{n,k}=\tfrac{4}{\beta k +1}\right)\ge 1-\frac{6}{\beta k},  \\[4pt] & P\left(\sigma_{n,k}>\tfrac{\log^{3+3/8} n}{\frac{\beta}{2} k}\right)\le 3 e^{-\frac15 \log^{1+1/8} n}.\label{signk}
\end{align}

Since $\sum_{k, n} P\left(\sigma_{n,k}>\tfrac{\log^{3+3/8} n}{\frac{\beta}{2} k}\right)\le 3 \sum_n n \, e^{-\frac13\log^{1+1/8} n}<\infty$, by the Borel-Cantelli  lemma there is a random $N_1\ge n_0$ so that for $n\ge N_1$ we have $\sigma_{n,k}\le \tfrac{\log^{3+3/8} n}{\frac{\beta}{2} k}$ a.s.
Set
\begin{align}\label{Znk}
Z_{n,k}=\frac{4}{\beta k+1}\ind(\sigma_{n,k}=\tfrac{4}{\beta k+1})+\tfrac{2 \log^{3+3/8} n }{\beta k} \ind(\sigma_{n,k}\neq \tfrac{4}{\beta k+1}).
\end{align}
For $n\ge N_1$ we have $ \sigma_{n,k}\le  Z_{n,k}$ and $\tau_{n,k}\le \sum_{j=k}^{n-1} Z_{n,j}$ a.s.

Next we will bound $P(\sum_{j=k}^{n-1} (Z_{n,j}-\tfrac{4}{\beta j+1})\ge \tfrac{\log^{9/2} n}{k})$. For $\lambda>0$ from (\ref{signk}) and  (\ref{Znk})  we get
\begin{align}\label{expmom}
E \left(e^{\sum_{j=k}^{n-1} \lambda (Z_{n,j}-\frac{4}{\beta k+1})}\right)\le \prod_{j=k}^{n-1} \left(1-\frac{6}{\beta j}+\frac{6}{\beta j}  e^{\lambda \tfrac{2 \log^{3+3/8} n}{\beta j}}    \right).
\end{align}
Assuming that $\lambda \tfrac{2\log^{3+3/8} n}{\beta k} \le 1$ we can use that $\tfrac{e^{x}-1}{x}\le 2$ for $x\le 1$ to bound the right side of (\ref{expmom}) as
\begin{align*}
 \prod_{j=k}^{n-1} \left(1-\frac{6}{\beta j}+\frac{6}{\beta j}  e^{\lambda \tfrac{2 \log^{3+3/8} n}{\beta j}}    \right)&\le  \prod_{j=k}^{n-1} \left(1+\frac{6}{\beta j}  \cdot 2 \lambda \frac{2 \log^{3+3/8} n}{\beta j}  \right)\le  \prod_{j=k}^{n-1} e^{\frac{24\lambda \log^{3+3/8} n }{\beta^2 j^2}}\\
 &\le e^{\frac{48\lambda\log^{3+3/8} n }{\beta^2 k}}.
\end{align*}
Setting now  $\lambda= \tfrac{\beta k}{2 \log^{3+3/8} n}$ and using the exponential Markov inequality we obtain
\begin{align*}
P\left(\sum_{j=k}^{n-1} (Z_{n,j}-\tfrac{4}{\beta j+1})\ge \tfrac{\log^{4+1/2} n}{ k}\right)&=E \left(e^{\sum_{j=k}^{n-1} \lambda (Z_{n,j}-\frac{4}{\beta k+1})}\right) e^{-\lambda  \tfrac{\log^{4+1/2} n}{ k}}
\\
&\le e^{\frac{48\lambda \log^{3+3/8} n }{\beta^2 k}-\lambda \tfrac{\log^{4+1/2} n}{ k}}\le e^{\frac{24}{\beta}-\frac{\beta}{2} \log^{1+1/8} n}.
\end{align*}
Since $\sum_{n} n   e^{\frac{24}{\beta}-\frac{\beta}{2} \log^{1+1/8} n}<\infty$, the Borel-Cantelli lemma implies that there is a random $N_0\ge N_1$ so that for $n\ge N_0$ and $\log^6 n\le k\le n$ we have a.s.
\[
\sum_{j=k}^{n-1} Z_{n,j}\le \sum_{j=k}^{n-1} \frac{4}{\beta j+1}+\tfrac{\log^{4+1/2}n}{k}\le \frac{4}{\beta} \log\left(\frac{n}{k}\right)+\tfrac{\log^{4+1/2} n}{k}.
\]
Since $\tau_{n,k}\le \sum_{j=k}^{n-1}$, the upper bound in (\ref{pathtime}) follows.
\end{proof}
\subsection{Path comparison}\label{s:pathcmp}

Let $\cB$ be a hyperbolic Brownian motion and consider the stopping times $\tau_{n,k}$ constructed in Proposition \ref{prop:stoptimes}. Set $\tilde \cB_n(t)=\cB(\tau_{n,\lceil (1-t) n\rceil})$ and $\tilde \cB(t)=\cB(-\tfrac{4}{\beta}\log(1-t))$.  The next proposition gives bounds on $d_{\HH}(\tilde \cB(t), \tilde \cB_n(t))$.

\begin{proposition}\label{prop:path} There is a random integer $N^*$   so that for $n\ge N^*$ we have the following a.s.~inequalities:
\begin{align}\label{path1}
d_\HH(\tilde \cB(t),\tilde \cB_n(t))&\le  \frac{\log^{3-1/8} n}{\sqrt{(1-t) n}}, \qquad \text{if} \quad 0\le t\le 1-\frac{\log^6 n}{n},\\
d_\HH(\tilde \cB_n(1-\tfrac{\log^6 n}{n}),\tilde \cB_n(t))&\le \frac{144}{\beta}(\log \log n)^2, \qquad \text{if} \quad 1-\frac{\log^6 n}{n}\le t<1.\label{path2}
\end{align}
\end{proposition}
\begin{proof}Let $T_n=1-\frac{\log^6 n}{n}$.
Consider $N_0$ from the statement of Proposition \ref{prop:stoptimes}. For $n\ge N_0$ and $0\le t\le T_n$ we have a.s.
\begin{align}\label{tcd1}
\left|\tau_{n,\lceil (1-t) n\rceil}-\tfrac{4}{\beta}\log(\tfrac1{1-t})   \right|&\le \tfrac{4}{\beta^2 \lceil (1-t) n\rceil}+\tfrac{\log^{9/2} n}{\lceil (1-t) n\rceil}+\tfrac{4}{\beta}\left|\log\left(\tfrac{n}{\lceil (1-t) n\rceil}\right) -\log(\tfrac{1}{1-t}) \right|\le \frac{\log^{4+5/8} n}{(1-t) n},
\end{align}
where for the second inequality we also assume $n\ge n_0$ with an $n_0$ only depending on $\beta$. Inequality (\ref{tcd1}) implies
\begin{align*}
d_\HH(\tilde \cB(t),\tilde \cB_n(t))\le \max\{d_\HH(\cB(s),\cB(s+u)): |u|\le h, \,0\le s+u \},
\end{align*}
with $h= \frac{\log^{4+5/8} n}{(1-t) n}$ and $s=\tfrac{4}{\beta}\log(\tfrac1{1-t})$. Note that $h\le \frac{\log^{4+5/8} n}{\log^6 n}\le \tfrac1{\log n}$.

Consider the random constant $h_0$ from the statement of  Proposition \ref{prop:modcont} of the Appendix. If $n\ge e^{h_0^{-1}}$ then $h\le h_0$  and we may apply Proposition \ref{prop:modcont} with $s, s+u$ if $0\le u\le h$ and with $s+u, s$ if $0\le -u\le \min(h,s)$.  Using the fact that $h \log(2+ \tfrac{s+1}{h})$ is monotone increasing in $h$ and $s$ we get
\begin{align*}
\max\{d_\HH(\cB(s),\cB(s+u)): |u|\le h, \,0\le s+u \}\le  20 \sqrt{h \log(2+ \tfrac{s+1}{h})}\le \frac{\log^{3-1/8} n}{\sqrt{(1-t) n}},
\end{align*}
if  $0\le t\le T_n$ and $n\ge N_1$, with a  random integer $N_1$.

%
%
%
%

To get the estimate for the $1-T_n\le t<1$ case we use 
%
the construction given in the proof of Proposition \ref{prop:stoptimes} and the triangle inequality repeatedly to get
\[
d_{\HH}(\tilde \cB_n(T_n),\tilde \cB_n(t))\le\sum_{j=1}^{\log^6 n-1} d_\HH(\cB_n(\tfrac{n-j}{n}), \cB_n(\tfrac{n-j-1}{n}))= \sum_{j=1 }^{\log^6 n-1} Y_j.
\]
Here $Y_j=\log\left(\tfrac{1+\sqrt{\xi_j}}{1-\sqrt{\xi_j}}  \right)$ and $\xi_j\sim$ Beta$(1,\frac{\beta}{2} j)$ are independent random variables. We have
\[
P(Y_j\ge r)=\sech^{\beta j}(r/2)\le 2e^{-\frac{\beta}{2} j r}
\]
which implies  $\sum_j P(Y_j\ge \tfrac{4}{\beta} \tfrac{\log j}{j})<\infty$. This means that there is a random $M_0$ so that for $m\ge M_0$ we have $Y_m\le \tfrac{4}{\beta} \tfrac{\log m}{m}$ and also $\sum_{j=1}^m Y_j\le \tfrac{4}{\beta} \log^2 m$. Thus if $\log^6 n\ge M_0$  then for $1-T_n\le t<1$ we have almost surely
\begin{align}
d_{\HH}(\tilde \cB_n(T_n),\tilde \cB_n(t))\le \frac{144}{\beta} (\log \log n)^2,\end{align}
%
which completes the proof of  (\ref{path2}).
\end{proof}

\section{Hilbert-Schmidt bounds} \label{s:HS}

%
%

This section contains general bounds on Dirac operators in terms of their defining paths.

\begin{proposition}[Hilbert-Schmidt property]\label{prop:HS1}
Let $\{\gamma(t), 0\le t\}$ be a measurable path in $\HH$,  $\eta_0, \eta_1\in \partial \HH$ distinct boundary points and $z_0\in \HH$. For a $\nu>0$ set $\tilde \gamma(t)=\gamma(\nu \log(\tfrac{1}{1-t}))$. Let $z(t)$ be the point moving with  speed $\alpha>0$ on the geodesic connecting $z_0$ to $\eta_1$ with $z(0)=z_0$.
Assume that there are constants  $b>0$ and $\eps\in(0,\nu^{-1})$ so that
\begin{align}\label{linear2}
d_{\HH}(\gamma(t), z(t))\le b+\eps t.
\end{align}
Then  the operator
\begin{align}\label{tauDir}
\tau=\Dirop(\tilde \gamma(t), \eta_0, \eta_1)
\end{align}
is self-adjoint on the appropriate domain and $\tau^{-1}$ is Hilbert-Schmidt.
\end{proposition}

\begin{proof}
We will check that the conditions of Theorem \ref{thm:Dirop} are satisfied.

If $Q$ is an isometry of $\HH$ then $d_{\HH}(z_1,z_2)=d_{\HH}(Q z_1, Q z_2)$ and $d_{\eta}(z_1,z_2)=d_{Q\eta}(Qz_1,Qz_2)$. Take an isometry $Q$ for which $Qz_0=i$ and $Q \eta_1=\infty$ and  denote $Q\eta_0$ by $q$. The  geodesic $z(t)$ is mapped into the geodesic connecting $i$ with $\infty$ with speed $\alpha$, thus $Qz(t)=i e^{ \alpha t}$.
From (\ref{horo}) it follows that
\begin{align*}
d_{\eta_1}(z(t),z_0)&=d_{\infty}(i e^{\alpha t},i)=-\alpha t,\\
d_{\eta_0}(z(t),z_0)&=d_q(i e^{\alpha t},q)=\alpha t-\log(1+q^2)\le \alpha t.
\end{align*}
From the triangle inequality we get
\begin{align}\notag
d_{\eta_1}(\gamma(t),z_0)&\le d_{\eta_1}(z(t),z_0)+d_\HH(\gamma(t), z(t))\le -(\alpha-\eps)t+b,\\
d_{\eta_0}( \gamma(t),z_0)&\le d_{\eta_0}(z(t),z_0)+d_\HH(\gamma(t), z(t))\le (\alpha+\eps)t+b.\label{pathbnd1}
\end{align}
The bounds in (\ref{horo_int1}) now follow easily:
\[
\int_0^1 e^{d_{\eta_1}(\tilde \gamma(t),z_0)}dt\le \int_0^1 e^{-(\alpha-\eps) \nu \log\left(\frac{1}{1-t}\right)+b}dt=e^b \int_0^1 (1-t)^{(\alpha-\eps) \nu} dt=\frac{e^b}{1+(\alpha-\eps)\nu}<\infty
\]
and
\begin{align*}
\int_0^1 \int_0^t  e^{d_{\eta_0}(\tilde \gamma(s),z_0)+d_{\eta_1}(\tilde \gamma(t),z_0)} ds dt&\le e^{2b} \int_0^1 \int_0^t (1-s)^{-(\alpha+\eps)\nu} (1-t)^{(\alpha-\eps)\nu}ds dt\\&= \frac{e^{2b}}{2(1+(\alpha-\eps)\nu)(1-\eps \nu)}<\infty.
 \end{align*}
%
%
%
%
\end{proof}

 Recall from (\ref{HSkernel}) that $\res \tau$ is an integral operator with kernel $K_{\res \tau}$. For $0<T<1$ we denote by $\ress_T \tau$ the integral operator with kernel $K_{\res \tau}(x,y)\cdot \ind(0\le x\le T, 0\le  y\le T) $.

\begin{proposition}[Hilbert-Schmidt truncation]\label{prop:HS2}
Let $\gamma, \eta_0, \eta_1, z_0, \alpha, \nu, \tilde \gamma, z(t)$ be as in Proposition \ref{prop:HS1},  and define $\tau$ according to (\ref{tauDir}).
\begin{enumerate}
\item Assume that (\ref{linear2}) holds for some $b>0$,  $\eps\in(0,\nu^{-1})$ with $\alpha \neq \nu^{-1}-\eps$. Then for any $T>0$ we have
\begin{align}\label{HStail1}
\|\res \tau-\ress_T \tau\|_{\HS}^2 \le C\,(1-T)^{1 +\min(\nu (\alpha-\eps),1-2\eps \nu)}.
\end{align}
with $C$ depending on $\eta_0, \eta_1, 1-\nu \eps, \alpha, b$.

\item Assume that for some $0<T<1$ (\ref{linear2}) holds for $0\le t\le \nu \log(\tfrac{1}{1-T})$ with some $b>0$ and $\eps\in(0,\nu^{-1})$  where $\alpha \neq \nu^{-1}-\eps$. Assume further that $d_\HH(\tilde \gamma(t),\tilde \gamma(T))\le M$ for $t\ge T$.
Then  we have
\begin{align}\label{HStail2}
\|\res \tau-\res \tau_T\|_{\HS}^2 \le C e^{2M} \, (1-T)^{1 +\min(\nu (\alpha-\eps),1-2\eps \nu)}.
\end{align}
with $C$ depending on $\eta_0, \eta_1, 1-\nu \eps, \alpha, b$.
\end{enumerate}
%
%
\end{proposition}
\begin{proof}
We denote the representation of $\gamma$ in the half plane by $x+i y$ and use $\tilde x+i \tilde y$ for the representation of $\tilde \gamma$. We represent $\eta_0, \eta_1$ with
nonzero vectors $u_0, u_1$ that satisfy $u_0^t J u_1=1$. Recall the integral kernel of $\res \tau$ from (\ref{HSkernel}). From the definition of $\ress_T \tau$ we get
\begin{align*}
2\|\res \tau-\ress_T \tau\|_{\HS}^2&=\int_T^1  \int_s^1  |a(s)|^2 |c(t)|^2 dt ds+ \int_0^T \int_T^1  |a(s)|^2 |c(t)|^2 ds dt,
\end{align*}
where $a(s)=X(s)u_0$ and $c(s)= X(s) u_1$ with $X=\tfrac{1}{\sqrt{\tilde y}}\mat{1}{-\tilde x}{0}{\tilde y}$.

From (\ref{horo}) one can check that if   $u\in \R^2$ is a nonzero vector then $e^{d_u(x+i y,i)}=\frac{|X u|^2}{|u|^2}$.
Using the triangle inequality we get
\begin{align*}
|a(s)|^2&=|X(s)u_0|^2=|u_0|^2 e^{d_{\eta_0}(\tilde \gamma,i)}\le |u_0|^2 e^{d_{\eta_0}(\tilde \gamma,z_0)+d_{\HH}(z_0,i)},\\
|c(s)|^2&=|X(s)u_1|^2=|u_1|^2 e^{d_{\eta_1}(\tilde \gamma,i)}\le |u_1|^2 e^{d_{\eta_1}(\tilde \gamma,z_0)+d_{\HH}(z_0,i)}.
\end{align*}
Using now (\ref{pathbnd1}) and recalling $\tilde \gamma(t)=\gamma(\nu \log(\tfrac{1}{1-t}))$ we get that if  (\ref{linear2}) holds for all $t$ then
%
%
\begin{align}\label{HSacbounds}
|c(t)|^2\le C  (1-t)^{(\alpha-\eps)\nu}, \quad |a(t)|^2\le C (1-t)^{-(\alpha+\eps)\nu},
\end{align}
where $C$ depends on $u_0, u_1$ and $b$.
This leads to
\begin{align*}
2\|\res \tau-\ress_T \tau\|_{\HS}^2&\le  C^2 \int_T^1  \int_s^1  (1-t)^{(\alpha-\eps)\nu}  (1-s)^{-(\alpha+\eps)\nu} dt ds \\[5pt]&\qquad+C^2 \int_0^T \int_T^1(1-t)^{(\alpha-\eps)\nu}  (1-s)^{-(\alpha+\eps)\nu}ds
dt\\[5pt]
&=C^2\frac{(1-T)^{2-2\eps \nu}}{2(1+(\alpha-\eps)\nu)(1-\eps \nu)}+C^2\frac{(1-T)^{(\alpha-\eps)\nu+1 } \left((1-T)^{-\nu  (\alpha +\eps )+1}-1\right)}{(1+(\alpha-\eps)  \nu) (1-\nu  (\alpha +\eps ))}.
\end{align*}
Considering the $\nu^{-1}-\eps-\alpha>0$ and $\nu^{-1}-\eps-\alpha<0$ cases separately we get (\ref{HStail1}).

Now assume that (\ref{linear2}) holds for $0\le t\le \nu \log(\tfrac{1}{1-T})$, and $d_\HH(\tilde \gamma(t),\tilde \gamma(T))\le M$ for $t\ge T$. Then for $0\le t\le T$ we still have (\ref{HSacbounds}), while for $t\ge T$ we can use $d_\eta(x,y)\le d_\eta(x,z)+d_\HH(y,z)$ to get
\[
|c(t)|^2\le C e^{M}  (1-T)^{(\alpha-\eps)\nu}, \qquad |a(t)|^2\le C e^{M} (1-T)^{-(\alpha+\eps)\nu}, \quad \text{for $T\le t<1$.}
\]
This gives
\begin{align*}
2 \|\res \tau-\res\tau_T\|_{\HS}^2&\le C^2 e^{2M} \int_T^1  \int_s^1  (1-T)^{-2\eps \nu} dt ds \\[5pt]&\qquad+C^2 e^{2M}\int_0^T  (1-s)^{-(\alpha+\eps)\nu} ds\int_T^1(1-T)^{(\alpha-\eps)\nu} dt
\\[5pt]
&= C^2 e^{2M} (1-T)^{2-2\eps\nu}+C^2 e^{2M}\frac{(1-T)^{(\alpha-\eps)\nu+1 } \left((1-T)^{-\nu  (\alpha +\eps )+1}-1\right)}{1-\nu  (\alpha +\eps )},
\end{align*}
from which (\ref{HStail2}) follows.
\end{proof}

\begin{proposition}[Hilbert-Schmidt approximation]\label{prop:HS3}
Let $\gamma, \eta_0, \eta_1, z_0, \alpha, \nu, \tilde \gamma,  z(t)$ be as in Proposition \ref{prop:HS1},  and define $\tau$ according to (\ref{tauDir}). Assume that for some $0<T<1$ (\ref{linear2}) holds for $0\le t\le \nu \log(\tfrac{1}{1-T})$ for some $b>0$ and $0<\eps<\min(\alpha,\tfrac1{2\nu})$.

Suppose that the path $\tilde \gamma_1$ is measurable and for $0\le t\le \nu \log(\tfrac1{1-T})$ we have
\begin{align}\label{distgam}
\sinh(\tfrac12 d_{\HH}(\tilde \gamma(t), \tilde \gamma_1(t)))^2\le \min(\delta (1-t)^{-1}, M) 
\end{align}
for some $M, \delta>0$.

Consider $\tau_1=\Dirop(\{\tilde \gamma_1(t), t\ge 0\}, \eta_0, \eta_1)$,  define $\ress_T \tau$  as in Proposition \ref{prop:HS2}, and  $\ress_T  \tau_{1}$ similarly.
Then
\begin{align}\label{hs3}
\|\ress_T \tau-\ress_T \tau_{1}\|_{\HS}^2\le C (M+1) \delta.
\end{align}
with a constant $C$ depending only on $\eta_0, \eta_1, \alpha, \nu, \eps, b$.
\end{proposition}

\begin{proof}
Denote the representation of $\tilde \gamma_1$ in the half plane by $\tilde x_1+i \tilde y_1$.

The hyperbolic distance formula in the upper half plane representation gives
\[
4\sinh(\tfrac12 d_{\HH}(\tilde \gamma, \tilde \gamma_1))^2= \frac{(\tilde x-\tilde x_1)^2}{ \tilde y \tilde y_1}+ \left(\sqrt{\tfrac{\tilde y}{\tilde y_1}}-\sqrt{\tfrac{\tilde y_1}{\tilde y}}\right)^2.
\]
Consider $a, c$, defined as in the proof of Proposition \ref{prop:HS2} and the analogously defined $a_1, c_1$.

If  $u\in \R^2$ is a nonzero vector, $X=\tfrac{1}{\sqrt{\tilde y}}\mat{1}{-\tilde x}{0}{\tilde y}$ and $X_1=\tfrac{1}{\sqrt{\tilde y_1}}\mat{1}{-\tilde x_1}{0}{\tilde y_1}$ then
\begin{align*}
\frac{|X u-X_1 u|}{|X u|}=\frac{|(I-X_1X^{-1})X u|}{|X u|}\le \|I-X_1 X^{-1}\|_2^2.
\end{align*}
An explicit computation gives
\begin{align*}
 \|I-X_1 X^{-1}\|_2^2&=\left(\frac{x-x_1}{\sqrt{\tilde y \tilde y_1}}\right)^2+\left(1-\sqrt{\frac{\tilde y_1}{\tilde y}}\right)^2+\left(1-\sqrt{\frac{\tilde y}{\tilde y_1}}\right)^2\\
 &\le\left(\frac{\tilde x-\tilde x_1}{\sqrt{\tilde y \tilde y_1}}\right)^2+\left(\sqrt{\tfrac{\tilde y}{\tilde y_1}}-\sqrt{\tfrac{\tilde y_1}{\tilde y}}\right)^2=4\sinh(\tfrac12 d_{\HH}(\tilde x+i \tilde y, \tilde x_1+i \tilde y_1))^2.
 \end{align*}
This yields
\begin{align*}
\frac{|c- c_1|^2}{|c|^2}\le 4\sinh(\tfrac12 d_{\HH}(\tilde \gamma, \tilde \gamma_1))^2,\qquad
\frac{|a- a_1|^2}{|  a|^2}\le4\sinh(\tfrac12 d_{\HH}(\tilde \gamma, \tilde \gamma_1))^2.
\end{align*}
To compute $\|\ress_T \tau-\ress_T \tau_{1}\|_{\HS}^2$ we need to estimate $\tr(\Delta K(s,t) \Delta K(s,t)^t)$ where $\Delta K(s,t)=\tfrac12 \left(a(s) c(t)^t- a_1(s)  c_1(t)^t\right)$. Using the Cauchy-Schwarz inequality we obtain
\begin{align*}
&4 \tr(\Delta K(s,t) \Delta K(s,t)^t)\\
&\qquad\le 3 \left( | a(s)|^2 | |c(t)- c_1(t)|^2+ |a(s)- a_1(s)|^2 |c(t)|^2+ |a(s)- a_1(s)|^2  |c(t)- c_1(t)|^2\right).
\end{align*}
The previous estimates with (\ref{distgam}) yield
\begin{align*}
&\|\ress_T \tau-\ress_T \tau_{1}\|_{\HS}^2= 2  \int_0^{T} \int_0^t   \tr(\Delta K(s,t) \Delta K(s,t)^t) ds dt\\
&\qquad \le 24 \int_0^{T} \int_0^t |a(s)|^2 |c(t)|^2( \sinh(\tfrac{d_H(s)}{2})^2 +\sinh(\tfrac{d_H(t)}{2})^2+ \sinh(\tfrac{d_H(s)}{2})^2 \sinh(\tfrac{d_H(t)}{2})^2 ) ds dt\\
&\qquad \le 24 \int_0^{T} \int_0^t |a(s)|^2 |c(t)|^2( \delta (1-t)^{-1}+\delta (1-s)^{-1}+\delta^2 (1-t)^{-1} (1-s)^{-1}) ds dt.
\end{align*}
Using the arguments in the proof of Proposition \ref{prop:HS2} we get that (\ref{HSacbounds}) holds for $0\le t\le T$, with a constant $C_1$ depending on $u_0, u_1$ and $b$.  Using this together with the bound $\delta (1-t)^{-1}\le M$ (which holds for  $t\in [0,T]$  by assumption) we get
\begin{align*}
&\|\ress_T \tau-\ress_T \tau_{1}\|_{\HS}^2\\
&\quad \le  24 C_1^2 \int_0^{T} \int_0^t (1-s)^{-(\alpha+\eps)\nu} (1-t)^{(\alpha-\eps)\nu}
( \delta (1-t)^{-1}+\delta (1-s)^{-1}+\delta^2 (1-t)^{-1} (1-s)^{-1}) ds dt\\
&\quad \le 24 C_1^2 \delta (1+\tfrac{M}{2})\int_0^{T} \int_0^t  \left((1-s)^{-(\alpha+\eps)\nu-1} (1-t)^{(\alpha-\eps)\nu}+(1-s)^{-(\alpha+\eps)\nu} (1-t)^{(\alpha-\eps)\nu-1}   \right)ds dt\\
&\quad \le24 C_1^2  \delta (2+M)\tfrac{1}{(\alpha-\eps)\nu (1-2\eps\nu)} .
\end{align*}
The bound (\ref{hs3})  now follows with the choice  $C= \tfrac{48 C_1^2}{(\alpha-\eps)\nu (1-2\eps\nu)} $.
\end{proof}

\section{Proof of Theorem \ref{thm:main}} \label{s:proof}

We can now return to the proof of our main theorem.
\begin{proof}[Proof of Theorem \ref{thm:main}]
Let $\cB$ be a hyperbolic Brownian motion. Set $\tilde \cB(t)=\cB(-\tfrac4{\beta}\log(1-t))$ for $0\le t<1$. Set $\eta_0=\infty$, $\eta_1=\cB(\infty)$ and set $\Sineop=\Dirop(\tilde \cB, \eta_0, \eta_1)$. 

Let $U$ be a uniform random variable on $[0,1]$ independent of $\cB$ and let  $\tau_{n,k}$ be  the stopping times constructed in Proposition \ref{prop:stoptimes}. According to the proposition the path $\tilde \cB_n(t)=\cB(\tau_{n,\lceil (1-t)n \rceil}), t\in [0,1)$ has the same distribution as the path in the construction of $\Circop$ in Theorem \ref{thm:circop}, and we may write $\Circop=\Dirop(\tilde \cB_n, \eta_0, \eta_1)$.

To prove the bound (\ref{mainbound})  we set $T_n=1-\frac{\log^6 n}{n}$. Recall the definition of $\ress_T$ from Section \ref{s:HS}  and write
\begin{align}\label{SC}
 \|\res \Sineop - \res \Circop\|^2_{\HS}&\le 3  \|\ress_{T_n} \Sineop - \ress_{T_n} \Circop\|^2_{\HS}\\&\quad +3  \|\res \Sineop - \ress_{T_n} \Sineop\|^2_{\HS}+3  \|\res \Circop - \ress_{T_n} \Circop\|^2_{\HS}. \notag
\end{align}
We will use Propositions \ref{prop:path}, \ref{prop:HS2} and \ref{prop:HS3} to estimate the three terms on the right.
Let $z(t)$ be the point moving with speed $1/2$ on the geodesic connecting   $\cB(0)$ to $\cB(\infty)$.  From Lemma \ref{lem:hypBMspeed} of the Appendix  it follows that for any $\eps>0$ there is a random $b=b_\eps$ so that a.s.~for all $t\ge 0$ we have
\begin{align}\label{Bz}
d_\HH(\cB(t), z(t))\le \eps t+b.
\end{align}
Applying the first statement of Proposition \ref{prop:HS2} for $\Sineop=\Dirop(\tilde \cB(t),  \eta_0, \eta_1)$ with $\gamma=\cB$, $\nu=\tfrac4{\beta}$, $\alpha=1/2$ and an $\eps>0$ satisfying $0<\eps <\min(\frac1{4}, \frac{\beta}{16})$ we get
\[
 \|\res \Sineop - \ress_{T_n} \Sineop\|^2_{\HS}\le C \left(\frac{\log^6 n}{n}\right)^{1+\min(\tfrac1{\beta},\tfrac12)}
\]
with a random $C$ depending on $\cB$ and $\beta$.

Recall from Proposition \ref{prop:path} that for $n\ge N^*$ we have
\begin{align}\label{pathtB1}
d_\HH(\tilde \cB(t),\tilde \cB_n(t))&\le  \frac{\log^{3-1/8} n}{\sqrt{(1-t) n}}\le 1, \qquad \text{if} \quad 0\le t\le T_n,\\
d_\HH(\tilde \cB_n(T_n),\tilde \cB_n(t))&\le \frac{144}{\beta}(\log \log n)^2, \qquad \text{if} \quad T_n\le t<1.\label{pathtB2}
\end{align}
Let $\cB_n(t)=\tilde \cB_n( 1-e^{-\tfrac{\beta}{4} t})$, then $\tilde \cB_n(t)=\cB_n(\tfrac{4}{\beta}\log(\tfrac1{1-t}))$ for $0\le t<1$.
From (\ref{Bz}), (\ref{pathtB1}) and the triangle inequality we get
\[
d_\HH(\cB_n(t), z(t))\le \eps t+b+1, \text{ for $0\le t\le \tfrac{4}{\beta}\log(\tfrac1{1-T_n})$.}
\]
Recall that  $\Circop=\Dirop(\tilde \cB_n(t),  \eta_0, \eta_1)$.  Applying the second statement of Proposition \ref{prop:HS2} with $\gamma=\cB_n$, $\nu=\tfrac4{\beta}$, $\alpha=1/2$, $M=\frac{144}{\beta}(\log \log n)^2$, $T=T_n$ and an $\eps>0$ satisfying $0<\eps <\min(\frac1{4}, \frac{\beta}{16})$ we get
\[
 \|\res \Circop - \ress_{T_n} \Circop\|^2_{\HS} \le  C e^{\frac{288}{\beta}(\log \log n)^2} \left(\frac{\log^6 n}{n}\right)^{1+\min(\tfrac1{\beta},\tfrac12)} \qquad \text{ if $n\ge N^*$,}
\]
with a random $C$ depending on $\cB$ and $\beta$.

Since $\sinh(x/2)^2\le  x^2$ for $0\le x\le 1$,  from (\ref{pathtB1}) we get
\[
\sinh(\tfrac12 d_H( \cB(t), \cB_n(t)) )^2\le  \frac{\log^{6-1/4} n}{n} (1-t)^{-1}  \le 1, \quad \text{if} \quad 0\le t\le T_n.
\]
Hence we may use Proposition \ref{prop:HS3} with $\gamma=\cB$, $\gamma_1=\cB_n$, $\nu=\tfrac4{\beta}$, $\alpha=1/2$, $T=T_n$, $\delta= \frac{\log^{6-1/4} n}{n}$, $M=1$ and $0<\eps<\min(\tfrac12, \tfrac{\beta}{2})$ to get
\[
\|\ress_{T_n} \Sineop - \ress_{T_n} \Circop\|^2_{\HS}\le C  \frac{\log^{6-1/4} n}{n}\quad \text{ if $n\ge N^*$,}
\]
with a random   $C$ depending on $\cB$ and $\beta$. Collecting our estimates, going back to (\ref{SC}) and modifying the random lower bound $N^*$ appropriately we get
\[
 \|\res \Sineop - \res \Circop\|^2_{\HS}\le \frac{\log^6 n}{n}, \qquad \text{for $n\ge N^*$}.
\]
This completes the proof of Theorem \ref{thm:main}.
\end{proof}

Before proving the three statements of Corollary \ref{cor:cpl} we state a law of large numbers for the points of $\Sineb$. 
\begin{proposition}\label{prop:LLN}
Suppose that the points of $\Sineb$ are given by $\lambda_k, k\in \ZZ$  with $\lambda_0<0<\lambda_1$. Then with probability one we have
\begin{align}\label{sineLLN}
\lim_{k\to \infty} \frac{\lambda_k}{k}=\lim_{k\to -\infty} \frac{\lambda_k}{k}=2\pi.
\end{align}
\end{proposition}
\begin{proof}
We will show the statement for $k\to \infty$, by the symmetry of the $\Sineb$ process this is sufficient. In  \cite{HV} it was shown that 
 $\tfrac1{\lambda} \#\{\lambda_k: 0\le \lambda_k\le \lambda\}$ satisfies a large deviation principle as $\lambda\to \infty$ with scale $\lambda^2$ and  rate function  $\beta I(\rho)$, where  $I(\frac{1}{2\pi})=0$ is the global minimum. From this the statement follows by a standard Borel-Cantelli argument. 
\end{proof}

\begin{proof}[Proof of Corollary \ref{cor:cpl}]
The bound (\ref{HSev}) follows directly from (\ref{mainbound}) and the Hoffman-Wielandt inequality for compact integral operators (see e.g.~\cite{BhatiaElsner}).

From Proposition \ref{prop:LLN} we see that there is a random constant $C$ so that a.s.~$|\lambda_k|\le C (|k|+1)$ for all $k$.
Set $a_n=\frac{n^{1/4}}{\log^{2} n}$. From (\ref{mainbound}) it follows that for large enough $n$ we have
\begin{align}\label{000}
\sup_{k} |\lambda_k^{-1}-\lambda_{k,n}^{-1}|\le \sqrt{\sum_{k} |\lambda_k^{-1}-\lambda_{k,n}^{-1}|^2}\le \frac{\log^{3} n}{n^{1/2}}.
\end{align}
Let $b_{k,n}=\lambda_k^{-1}-\lambda_{k,n}^{-1}$, then
\begin{align}\label{egy}
(\lambda_{k,n}-\lambda_k)(1-b_{k,n}  \lambda_k)=b_{k,n}  \lambda_k^2.
\end{align}
For large enough $n$ we have
\[
\max_{|k|\le a_n} |b_{k,n} \lambda_k| \le  \frac{\log^{3} n}{n^{1/2}} C (a_n+1)=C\left( \frac{\log n}{n^{1/4}}+ \frac{\log^{3} n}{n^{1/2}}\right)
\]
which means that $\lim\limits_{n\to \infty} \max_{|k|\le a_n} |b_{k,n} \lambda_k|=0$ a.s. From (\ref{egy}), for large enough $n$, we have
\[
\max_{|k|\le a_n}|\lambda_{k,n}-\lambda_k|\le 2 C^2 (a_n+1)^2 \max_{|k|\le a_n} |b_{k,n}| \le  4 C^2 a_n^2 \frac{\log^{3} n}{n^{1/2}}=\frac{4 C^2 }{\log n}.
\]
This completes the proof of (\ref{cpl}).

By Proposition \ref{prop:LLN} we may  choose a random $C>0$ so that $|\lambda_k| \le C \sqrt{1+k^2}$ for all $k$. Then we have $\frac{1}{C\sqrt{1+k^2}}\le \frac{1}{|\lambda_k|}$ and for $n\ge N$, $|k|\le n^{1/2-\eps}$ we have
\[
 \frac{1}{|\lambda_{k,n}|}\ge \frac{1}{|\lambda_k|}-\left|\frac{1}{\lambda_k}-\frac{1}{\lambda_{k,n}}   \right|\ge \frac{1}{C\sqrt{1+k^2}}-\frac{\log^3 n}{\sqrt{n}}\ge \frac{1}{2C \sqrt{1+k^2}}.
\]
(For the last inequality one might need to change $N$.) Then for such $k$ and $n$ we get
\[
|\lambda_k-\lambda_{k,n}|=\left|\frac{1}{\lambda_k}-\frac{1}{\lambda_{k,n}}   \right| |\lambda_k| |\lambda_{k,n}|\le \frac{\log^3 n}{n^{1/2}} 2 C^2 (1+k^2)\le \frac{1+k^2}{n^{1/2-\eps}},
\]
again by setting a large enough random lower bound on $n$. 
\end{proof}

\section{Estimating the dependence on $\beta$ in $\Sineop$}\label{s:beta}

The construction of the $\Sineop$ operator provides a natural coupling of this operator for all values of $\beta$.  The techniques we developed for the proof of Theorem \ref{thm:main} can be used to estimate how the operator $\Sineop$  (and the process $\Sineb$) depends on the value of $\beta$ in this coupling.

\begin{proof}[Proof of Theorem \ref{thm:betadep}]
We have
\[
\Sineop=\Dirop(\tilde \gamma, \eta_0, \eta_1), \qquad \mathtt{Sine}_{\beta_1}=\Dirop(\tilde \gamma_1, \eta_0, \eta_1)
\]
where $\tilde \gamma(t)=\cB(\tfrac{4}{\beta}\log(\tfrac1{1-t}))$, $\tilde \gamma_1(t)=\cB(\tfrac{4}{\beta_1} \log(\tfrac1{1-t}))$, $\eta_0=\infty$ and $\eta_1=\cB(\infty)$. 
We will estimate $ \|\res \Sineop-\res \mathtt{Sine}_{\beta_1}   \|_{\HS}  $ using Propositions \ref{prop:HS2} and \ref{prop:HS3}. Note that with probability one $\eta_0\neq \eta_1$ and hence a.s.~0 is not an eigenvalue for any of the $\Sineop$ operators.

Set $T=1-\delta<1$. Cutting off $\Sineop$ and $\mathtt{Sine}_{\beta_1}$ at $T$  using the first statement of Proposition \ref{prop:HS2} gives
\[
\|\res \Sineop-\ress_T \Sineop\|^2_{\HS}+\| \res \mathtt{Sine}_{\beta_1}-\ress_T \mathtt{Sine}_{\beta_1}\|^2_{\HS}\le C_0 \delta^{1+\min(4 \kappa,1)-\eps}
\]
for a small enough $\eps>0$ with a random $C_0$. (Recall that $\kappa\le \beta^{-1}<\beta_1^{-1}$.)

To estimate $\| \ress_T \mathtt{Sine}_{\beta}-\ress_T \mathtt{Sine}_{\beta_1}\|^2_{\HS}$ we first bound $d_\HH(\tilde \gamma(t), \tilde \gamma_1(t))$ in $[0,T]$. For 
 $0<t\le T$ we have $(\tfrac{4}{\beta_1}-\tfrac{4}{\beta})\log(\tfrac1{1-t})\le \delta \log(\tfrac{1}{\delta})<1$ and 
by Proposition \ref{prop:modcont} (and the comment following the proposition)   we get the bound
\begin{align}\label{dh}
d_\HH(\cB(\tfrac{4}{\beta}\log(\tfrac1{1-t}), \cB(\tfrac{4}{\beta_1}\log(\tfrac1{1-t}))^2<C^2 \delta \log(\tfrac1{1-t}) \log\left(2+\frac{\tfrac{4}{\beta }\log(\frac{1}{1-t})+1}{\delta \log(\tfrac1{1-t})}\right)
\end{align}
for $0\le t\le T$ with a random $C$.  

The right side of (\ref{dh}) can be bounded by
\begin{align}
C^2  \delta \log(\tfrac1{1-t}) \left(\log\left((\tfrac{4}{\kappa}+2)\log(\tfrac{1}{\delta})+1\right)+\log\left(\tfrac{1}{\delta \log(\frac1{1-t})}\right)   \right)
\end{align}
which can further be bounded by $C^2 c_1$ with a deterministic $c_1$ depending only $\kappa$. 

Note that since $\delta<1/3$, we have
\[
\log(\tfrac1{1-t}) \log\left((\tfrac{4}{\kappa}+2)\log(\tfrac{1}{\delta})+1\right)\le c_2 \log(\tfrac1{\delta} ) (1-t)^{-1}
\]
with a deterministic $c_2$ depending only on $\kappa$.  Moreover, it is a straightforward calculus exercise to show that for $0<t\le T=1-\delta$ we have
\[
 \delta \log(\tfrac1{1-t}) \log\left(\tfrac{1}{\delta \log(\frac1{1-t})}\right) \le e^{-1/3} \delta\log(\tfrac1{\delta}) (1-t)^{-1}.
\]
This gives the bound 
\[
d_\HH(\tilde \gamma(t), \tilde \gamma_1(t))^2<C_1\min(\delta \log(\tfrac1{\delta}) (1-t)^{-1},1)
\]
with a constant depending on $C$ and $\kappa$. We can turn this into an upper bound of the form 
\[
\sinh(\tfrac12 d_{\HH}(\tilde \gamma(t), \tilde \gamma_1(t))^2\le C_2 \min(\delta \log(\tfrac1{\delta}) (1-t)^{-1},1),
\]
from which Proposition \ref{prop:HS3} gives
\[
\| \ress_T \mathtt{Sine}_{\beta}-\ress_T\mathtt{Sine}_{\beta_1}\|^2_{\HS}\le C_3 \delta \log(\tfrac1\delta).
\]
Collecting all the terms gives
\[
\| \res \mathtt{Sine}_{\beta}-\res \mathtt{Sine}_{\beta_1}\|^2_{\HS}\le C  \delta \log(\tfrac1\delta)
\]
with a random $C$. The Hoffman-Wielandt inequality finishes  the proof of (\ref{sb1}).
\end{proof}

\section{Appendix} \label{s:app}

In the first part of the Appendix we collect the proofs of Lemmas  \ref{lem:betacpl} and \ref{lem:stdom1} that were used in the single step coupling of Proposition \ref{prop:1step}. In the second part we collect some estimates on the hyperbolic Brownian motion.

\subsection{Proof of the coupling statements} \label{s:calc}

 We now return to the proof of Lemma \ref{lem:betacpl}. Let $\cB$ be a hyperbolic Brownian motion and set $\zeta=\zeta_t=d_\HH(\cB(0),\cB(t))$.

The circumference of a hyperbolic circle of radius $r$ is $2\pi \sinh(r)$ and the density function of the hyperbolic BM at time $t$ and distance $r$ is given by the following formula (see e.g. \cite{KST}):
\[
g(r,t)=\frac{\sqrt{2} e^{-t/8}}{(2\pi t)^{3/2}}  \int_r^\infty \frac{s e^{-\frac{s^2}{2t}}}{\sqrt{\cosh s-\cosh r}} ds.
\]
From this the density of $\zeta_t$ is
\begin{align}
p_{\zeta}(r,t)=2\pi \sinh(r) g(r,t)=\frac{ e^{-t/8} \sinh(r)}{\sqrt{\pi} t^{3/2}}  \int_r^\infty \frac{s e^{-\frac{s^2}{2t}}}{\sqrt{\cosh s-\cosh r}} ds, \label{zetapdf}
\end{align}
and the tail of the cumulative distribution function  is
\begin{align}\notag
1-F_{\zeta}(r)&=\int_r^\infty \sinh(u)  	\frac{ e^{-t/8}}{\sqrt{\pi} t^{3/2}}  \int_u^\infty \frac{s e^{-\frac{s^2}{2t}}}{\sqrt{\cosh s-\cosh u}} ds\,	du\\ \notag
&=\int_r^\infty  \frac{ e^{-t/8}s e^{-\frac{s^2}{2t}}}{\sqrt{\pi} t^{3/2}} \int_r^s  \frac{\sinh(u)}{\sqrt{\cosh s-\cosh u}}du\, ds\\
&=\int_r^\infty  \frac{ 2 e^{-t/8}s e^{-\frac{s^2}{2t}}}{\sqrt{\pi} t^{3/2}}  \sqrt{\cosh (s)-\cosh (r)} ds. \label{zetacdf}
\end{align}
Let $Y=Y_\gamma=\log(\frac{1+\sqrt{\xi}}{1-\sqrt{\xi}})$ where $\xi$ has distribution \textup{Beta}$(1,\gamma)$. We record the cumulative distribution function $F_Y$ and the probability density function $p_Y$  of $Y$, which follow from direct computation of the Beta distribution:
\begin{align}\label{Ycdf}
F_Y(r)=1-\sech^{2\gamma}(\tfrac{r}{2}),\qquad
p_Y(r)=\gamma  \sinh \left(\frac{r}{2}\right) \text{sech}^{2 \gamma +1}\left(\frac{r}{2}\right).
\end{align}

We start with a simple estimate.
\begin{lemma}\label{lem:cosh}
For $0\le r\le s$ we have
\begin{align}\label{coshbnd}
\frac1{\sqrt{2}}\sqrt{s^2-r^2}\le \sqrt{\cosh(s)-\cosh(r)}\le \frac1{\sqrt{2}}\sqrt{s^2-r^2}\exp\left(\frac{r^2+s^2}{24}\right)
\end{align}
\end{lemma}
\begin{proof}The statement follows  from the bound $1\le \frac{\sinh(x)}{x}\le e^{x^2/6}$ and
\[
\frac{{2(\cosh(s)-\cosh(r))}}{s^2-r^2}=\frac{\sinh((s-r)/2)}{(s-r)/2} \cdot \frac{\sinh((s+r)/2)}{(s+r)/2}. \qedhere
\]
\end{proof}
Applying Lemma \ref{lem:cosh} to (\ref{zetapdf}) and (\ref{zetacdf}) and computing the resulting integrals directly we get the following bounds:
\begin{align}\label{pdf1}
p_\zeta(r,t)&\ge  (1+\tfrac{t}{12})^{-1/2}\, \frac1{t} e^{-\frac{r^2}{2 t}-\frac{r^2}{12}-\frac{t}{8}} \sinh (r)  =:p_{-}(r,t)\\ \label{pdf2}
p_\zeta(r,t)&\le \frac1{t} {e^{-\frac{r^2}{2 t}-\frac{t}{8}} \sinh (r)}=:p_{+}(r,t)\\\label{cdf1}
1-F_\zeta(r)&\le {(1-\tfrac{t}{12})^{-3/2}} e^{-\frac{r^2}{2 t}+\frac{r^2}{12}-\frac{t}{8}},
\end{align}
where the last bound is valid for $0<t<12$.

We also record the following bounds on $\log \cosh(r)$ which can be readily checked by differentiation and Taylor expansion.
\begin{align}\label{coshUB}
\log \cosh(x)&\le \tfrac{x^2}{2}, \qquad\qquad \text{for all $x$}\\
\log \cosh(x)&\ge \tfrac{x^2}{2}-\tfrac{x^4}{12}, \qquad \text{for $0\le x\le 1$}\label{coshLB}
\end{align}

\begin{proof}[Proof of (\ref{lem:dom}) in Lemma \ref{lem:betacpl}]
From (\ref{Ycdf}) and (\ref{coshUB}) we get
\[
1-F_Y(r)=\exp(-2\gamma \log \cosh(r/2))\ge \exp(-\gamma \tfrac{r^2}{4})=\exp(-\tfrac{r^2}{2t}+\tfrac{r^2}{8}).
\]
One can check that if $0< t\le 1$ and $t\le r$ then
\[
  {(1-\tfrac{t}{12})^{-3/2}} e^{-\frac{r^2}{2 t}+\frac{r^2}{12}-\frac{t}{8}}\le \exp(-\tfrac{r^2}{2t}+\tfrac{r^2}{8}).
\]
Together with (\ref{cdf1}) this proves the statement for $t\le r$.

To prove the statement in the $0\le r<t$ case, we will show
\[
\int_0^r p_\zeta(u,t) du\ge \int_0^r p_Y(u) du \qquad \textup{for $0\le r<t\le 1$}
\]
using the lower bound (\ref{pdf1}) on $p_\zeta(r,t)$. We will prove that for $0<r<t\le 1$ we have
\[
p_-(r,t)= (1+\tfrac{t}{12})^{-1/2}\, \frac1{t} e^{-\frac{r^2}{2 t}-\frac{r^2}{12}-\frac{t}{8}} \sinh (r) \ge \gamma  \sinh \left(\frac{r}{2}\right) \text{sech}^{2 \gamma +1}\left(\frac{r}{2}\right)=p_Y(r).
\]
The last inequality is equivalent to
\begin{align}\label{333}
 e^{-\frac{r^2}{2 t}-\frac{r^2}{12}} \ge e^{t/8} \sqrt{1+\tfrac{t}{12}} \left(1-\tfrac{t}{4}\right)\exp\left(-(\tfrac{4}{t}+1)\log \cosh(r/2)\right).
\end{align}
For $0\le t\le 1$ we have $e^{t/8} \sqrt{1+\tfrac{t}{12}} \left(1-\tfrac{t}{4}\right)\le 1$.  Then by  the bound (\ref{coshLB})
we get
\[
e^{t/8} \sqrt{1+\tfrac{t}{12}} \left(1-\tfrac{t}{4}\right) \exp\left(-(\tfrac{4}{t}+1)\log \cosh(r/2)\right)\le \exp\left(-(\tfrac{4}{t}+1)(\tfrac{r^2}{8}-\tfrac{r^4}{192})\right).
\]
The finish the proof of (\ref{333}) we need to show that for  $0<r<t\le 1$ we have
\[
\left(\frac{4}{t}+1\right)\left(\frac{r^2}{8}-\frac{r^4}{192}\right)\ge \frac{r^2}{2 t}+\frac{r^2}{12},
\]
which follows from direct computation.
\end{proof}
Now we turn to the proof of the total variation bound.
\begin{proof}[Proof of (\ref{lem:TV}) in Lemma \ref{lem:betacpl}]
We need to show that  $\int_0^\infty |p_Y(r)-p_\zeta(r,t)|dr\le 3t$.

By part (a) we have $P(\zeta>r)\le P(Y_\gamma>r)$ which leads to
\[
\int_K^\infty |p_Y(r)-p_\zeta(r,t)|dr\le 2 (1-F_Y(K))=2 \,\sech^{2\gamma}(\tfrac{K}{2}), \qquad \text{for $K>0$}.
\]
Setting $K=2 \sqrt{t \log(2/t)}\le 2$ we get
\[
\int_K^\infty |p_Y(r)-p_\zeta(r,t)|dr\le 2 \,\sech^{2\gamma}(\tfrac{K}{2})\le 2 \exp(- 2\gamma\cdot \tfrac{5}{12}(K/2)^2)\le 2 (t/2)^{\tfrac{5}{4}}\le t,
\]
where we used  (\ref{coshLB}) in the second inequality.

Using the triangle inequality and the bound (\ref{pdf1}) we get
\[
\int_0^K |p_Y(r)-p_\zeta(r,t)|dr\le \int_0^K |p_Y(r)-p_{+}(r,t)|dr+\int_0^K (p_{+}(r,t)-p_\zeta(r,t))dr.
\]
We can bound the second integral explicitly:
\begin{align*}
\int_0^K (p_{+}(r,t)-p(r,t))dr&\le \int_0^\infty (p_{+}(r,t)-p(r,t))dr=\frac{\sqrt{\frac{\pi }{2}} e^{\frac{3 t}{8}} \left(2 \Phi\left(\sqrt{t}\right)-1\right)}{\sqrt{t}}-1\\
&\le e^{\frac{3t}{8}}-1\le t/2.
\end{align*}
Introduce
\[
p_0(r,t)=\frac{2 e^{-\frac{r^2}{2 t}} \sinh \left(\frac{r}{2}\right)}{t}.
\]
Then
\[
\frac{p_+}{p_0}=e^{-t/8} \cosh(r/2), \qquad \frac{p_Y}{p_0}=(1-t/4)e^{\frac{r^2}{2 t}} \text{sech}^{\frac{4}{t}}\left(\tfrac{r}{2}\right).
\]
We have, for $0\le t\le 1$, $0\le r\le 2$:
\[
\left|\frac{p_+}{p_0}-1\right|\le \frac{t}8+\frac{r^2}4.
\]
Using the bounds (\ref{coshUB}) and  (\ref{coshLB}) we get that for $0\le r\le K<2$ we have
\[
1=\exp\left(\tfrac{r^2}{2 t}- \tfrac{4}{t} \tfrac{r^2}{8}\right) \le e^{\frac{r^2}{2 t}} \text{sech}^{\frac{4}{t}}\left(\tfrac{r}{2}\right)\le \exp\left(\tfrac{r^2}{2 t}- \tfrac{4}{t}(\tfrac{r^2}{8}-\tfrac{r^4}{192})\right)=e^{\frac{r^4}{48 t}}
\]
This leads to
\[
\left| \frac{p_Y}{p_0}-1\right|\le \frac{t}{4}+|e^{\frac{r^4}{48t}}-1|\le \frac{t}{4}+\frac{r^4}{24t}
\]
where we used that $\frac{r^4}{t}\le t \log(2/t)^2\le 1$ if $t\le 1$. Note also that $\sinh(x)\le \frac{6}{5}x$ for $x\le 1$.
From this
\[
\int_0^K |p_Y(r)-p_{+}(r,t)|dr\le \int_0^K (\tfrac38 t+\tfrac 14 r^2+\tfrac{r^4}{24 t}) p_0(r)dr\le \frac65 \int_0^\infty (\tfrac38 t+\tfrac 14 r^2+\tfrac{r^4}{24 t})\frac{ e^{-\frac{r^2}{2 t}} r }{t}dr=\frac{29}{20}t.
\]
Collecting all our estimates gives
\begin{equation*}
\int_0^\infty |p_Y(r)-p_\zeta(r,t)|dr\le 3t. \qedhere
\end{equation*}
\end{proof}
The proof of Lemma \ref{lem:stdom1} follows a standard coupling construction, we include it for completeness.
\begin{proof}[Proof of Lemma \ref{lem:stdom1}]
Denote the distributions of $X_1, X_2$ by $\mu_1, \mu_2$, and let $\mu=\tfrac12(\mu_1+\mu_2)$. Denote that density function of $X_i$ with respect to $\mu$ by $f_i$. From our assumptions it follows that
\[
\tilde f_0(x)=\tfrac{1}{1-\eps}\min(f_1(x),f_2(x)),   \quad \tilde f_1(x)=\tfrac{1}{\eps}|f_1(x)-f_{2}(x)|^+, \quad \tilde f_2(x)=\tfrac{1}{\eps}|f_2(x)-f_{1}(x)|^+,
\]
are also  density functions with respect to $\mu$, and the  distributions corresponding to $\tilde f_1, \tilde f_2$ are stochastically ordered just as $X_1$ and $X_2$. Moreover, $(1-\eps) \tilde f_0+\eps \tilde f_i=f_i$ for $i=1,2$.

Recall that if $F$ is a cumulative distribution function,  $F^{-1}(x)=\sup\{y: F(y)<x\}$ is its generalized inverse, and $U\sim$ Uniform$[0,1]$ then $F^{-1}(U)$ has cumulative distribution function given by $F$.

Denote the cumulative distribution function  corresponding to $\tilde f_i$ by $\tilde F_i$.
Let $U_1, U_2$ be independent uniform random variables on $[0,1]$ and consider the pair of random variables
\[
(Y_1,Y_2)=\ind(U_1\le 1-\eps) (\tilde F_0^{-1}(U_2), \tilde F_0^{-1}(U_2))+\ind(U_1> 1-\eps)  (\tilde F_1^{-1}(U_2), \tilde F_2^{-1}(U_2)).
\]
In plain words: with probability $1-\eps$ we generate $(Z,Z)$ where $Z$ has density $\tilde f_0$, and with probability $\eps$ we generate $(\tilde X_1, \tilde X_2)$ where $\tilde X_i$ has density $\tilde f_i$ and $\tilde X_1\le \tilde X_2$ a.s. Then $Y_1\le Y_2$ a.s., $P(Y_1=Y_2)=1-\eps$ and $Y_i$ has the same distribution as $X_i$.
Consider the regular conditional distribution of $Y_2$ given $Y_1$, and let $g(x,u)$ be the generalized inverse of the conditional cumulative distribution function  of $Y_2$ given $Y_1=x$. Then $(X_1, g(X_1,U))$ has the same joint distribution as $(Y_1,Y_2)$, and thus $g$ satisfies the requirements of the lemma.
\end{proof}

%
%
%
%
%

\subsection{Hyperbolic Brownian motion estimates}

The first two lemmas give estimates on the behavior of the process $d_\HH(\cB(0), \cB(t))$ where $\cB$ is a (standard) hyperbolic Brownian motion.
\begin{lemma}\label{lem:hypBMcpl}
There is a coupling of a  hyperbolic Brownian motion $\cB$, a one dimensional standard Brownian motion $B$ and a 2 dimensional standard Brownian motion $W$
 so that almost surely for all $t\ge 0$ we have
\[
|B(t)|\le d_\HH(\cB(0), \cB(t))\le |W(t)|+t/2.
\]
\end{lemma}
\begin{proof}
The process $q_t=d_\HH(0,\cB(t))$ satisfies the SDE
\[
dq=db+\frac{\coth q}{2}dt, \qquad q(0)=0,
\]
where $b$ is a standard Brownian motion.

Consider the following diffusions with the same driving Brownian motion $b$ as in $q$.
\[
dq_1=db+\frac{1}{2q_1}dt, \qquad dq_2=db+\left(\frac{1}{2q_2}+\frac12\right)dt, \qquad q_1(0)=q_2(0)=1.
\]
Since $\tfrac1x\le \coth x\le \tfrac1{x}+1$ for $x>0$, we have a.s.~$q_1\le q\le q_2$. The process $q_1$ is a 2 dimensional Bessel process, it can be written as the absolute value of a 2 dimensional Brownian motion $W=(B,B')$, which is at least as large as the absolute value of its first coordinate $|B|$.

For the upper bound note that we have $q_1\le q_2$, and taking the difference of the SDEs for $q_2$ and $q_1$ we get
\[
(q_2-q_1)'=\frac{1}{2q_2}-\frac{1}{2q_1}+\frac12\le \frac12.
\]
Thus $q(t)\le q_2(t)\le q_1(t)+t/2=|W(t)|+t/2$, which finishes the proof.
\end{proof}

\begin{lemma}\label{lem:hypBMtau}
Let $\cB$ be a standard hyperbolic Brownian motion. Then for $t>0, a>0$ we have
\begin{align}\label{lowerhBM}
 P(\max_{0\le s \le t}d_\HH(\cB(0),\cB(s))\le  a)\le \frac{4}{\pi} e^{-\frac{\pi^2 t}{8 a^2}}.
\end{align}
If $0<t\le a$ then
\begin{align}\label{upperhBM}
 P(\max_{0\le s \le t}d_\HH(\cB(0),\cB(s))\ge  a)\le \frac{16\sqrt{t}}{a \sqrt{\pi}}e^{-\frac{a^2}{16t}}.
 \end{align}
%
%
%
%
%
\end{lemma}
\begin{proof}
By Lemma \ref{lem:hypBMcpl} the process $q_t=d_\HH(\cB(0), \cB(t))$ stochastically dominates $|B(t)|$ where $B$ is a standard Brownian motion. Then
\begin{align*}
 P(\max_{0\le s \le t}d_\HH(\cB(0),\cB(s))\le  a)&\le  P(\max_{0\le s \le t}|B(s)|\le  a)\le \frac{4}{\pi} e^{-\frac{\pi^2 t}{8 a^2}},
\end{align*}
which proves (\ref{lowerhBM}).
The last step follows from the following identity (see e.g.~\cite{BMbook}):
\[
P(\max\limits_{0\le s \le t}|B(s)|\le u)=\frac4{\pi} \sum_{k=0}^\infty \frac{(-1)^k}{2k+1}e^{-\frac{(2k+1)^2\pi^2 t}{8 u^2}}.
\]

From Lemma \ref{lem:hypBMcpl} it follows that $q_t=d_\HH(\cB(0), \cB(t))$ is stochastically dominated by the process $|W(t)|+t/2$ where $W$ is a 2 dimensional standard Brownian motion. Thus
\begin{align*}
 P(\max_{0\le s \le t}d_\HH(\cB(0),\cB(s))\ge  a)&\le  P(\max_{0\le s \le t}(|W(s)|+s/2)\ge  a)\le  P(\max_{0\le s \le t}|W(s)|\ge  a-t/2)\\
 & \le 4 P(\max_{0\le s \le t}B(s)\ge  \tfrac{1}{\sqrt{2}}(a-t/2))=
 4 P(|B(t)|\ge  \tfrac{1}{\sqrt{2}}(a-t/2))
 \\&\le \frac{8\sqrt{t}}{(a-t/2)\sqrt{\pi}}e^{-\frac{(a-t/2)^2}{4t}}\le \frac{16\sqrt{t}}{a \sqrt{\pi}}e^{-\frac{a^2}{16t}}.
\end{align*}
In the second line we compared $W$ to a one dimensional standard Brownian motion $B$, then used the reflection principle. Finally we used the well known tail bound for the normal distribution and $t\le a$.
%
%
\end{proof}

The next lemma shows that $\cB$ approaches its limit point with speed $1/2$.
\begin{lemma}\label{lem:hypBMspeed}
Let $\cB$ be a hyperbolic Brownian motion and let $z(t)$ be the point moving with speed $1/2$ on the geodesic connecting $\cB(0)$ to $\cB(\infty)=\lim_{t\to \infty} \cB(t)$. Then almost surely
\[
\lim_{t\to \infty} \frac1{t}d_\HH(\cB(t), z(t))=0.
\]
\end{lemma}
\begin{proof}
Consider the half plane representation of $\HH$ where $\cB(0)=i$ and $\cB(\infty)=\infty$, and denote the representation of $\cB$ by $x+i y$. Then $x+i y$ is a hyperbolic Brownian motion conditioned to hit $\infty$, in particular it satisfies
\[
dy =y(dB_1+dt), \qquad dx=y dB_2, \qquad y(0)=1,\quad x(0)=0.
\]
where $B_1, B_2$ are independent standard Brownian motions. (See e.g.~\cite{BVBV_sbo}.)
The geodesic connecting $i$ and $\infty$ is $i e^{t}, t\ge 0$, and we have $z(t)=i e^{t/2}$. By the triangle inequality
\begin{align*}
d_\HH(\cB(t), z(t))&\le d_\HH(x+i y, i y)+d_\HH(i y, z(t))\\
&\le \operatorname{arccosh}\left(1+\tfrac{x^2}{2y^2}\right)+ |\log(y e^{-t/2})|\le \log(2+\tfrac{x^2}{y^2})+ |\log(y e^{-t/2})|.
\end{align*}
We can explicitly solve for $y$ and $x$ from the SDE: $y(t)=e^{B_1(t)+t/2}$, $x(t)=\int_0^t e^{B_1(s)+s/2}ds$.
Since $\frac1{t}|\log(y e^{-t/2})|=\frac{|B_1(t)|}{t}$ converges to 0 a.s., 
 for any $\delta>0$ we have
\begin{align}\label{geobm}
C_\delta^{-1} e^{(1/2-\delta) t}\le y(t)\le C_\delta e^{(1/2+\delta) t}
\end{align}
with a random $C_\delta$.

For the  term $ \log(2+\tfrac{x^2}{y^2})$ we start with the observation that there is a standard Brownian motion $W$ so that $x(t)=W(\int_0^t y(s)^2 ds)$.
Since $\lim_{t\to \infty}\int_0^t y(s)^2 ds= \infty$ by (\ref{geobm}), the iterated logarithm theorem implies that  for any $\eps>0$ we have
\[
\lim_{t\to \infty} \frac{x(t)}{(\int_0^t y(s)^2 ds)^{1/2+\eps}}=0.
\]
Using (\ref{geobm}) we get the following estimate for fixed $\eps, \delta>0$:
\[
\frac{{(\int_0^t y(s)^2 ds)^{1+2\eps}}}{y(t)^2}\le C' e^{t \left\{(1+2\delta)(1+2 \eps)-(1-2\delta)\right\}}=C' e^{(4 \delta  \eps +4 \delta +2 \eps)t}.
\]
This shows that $\lim_{t\to \infty} \tfrac1{t} \log(2+\tfrac{x(t)^2}{y(t)^2})= 0$ a.s.  and completes the proof of the lemma.
\end{proof}
%

%
%
%
%
%
%
The next statement gives an estimate on the modulus of continuity of the hyperbolic Brownian motion.
The proof follows that of the analogous statement for standard Brownian motion.  We include it for completeness. The constants have not been optimized. 
\begin{proposition}\label{prop:modcont}
Let $\cB$ be a hyperbolic Brownian motion. Then there is a random  constant $0<h_0\le 1$ so that a.s.
\begin{align}\label{modcon}
d_{\HH}(\cB(s),\cB(s+h))\le 20 \sqrt{h \log(2+ \tfrac{s+1}{h})},
\end{align}
for all $0<h\le h_0$ and $0\le s$.
\end{proposition}
Note that the  proof below also shows that there is a random constant $C$ so that (\ref{modcon}) holds for all $0<h\le 1$ with $C$ in place of 20.
\begin{proof}
Let $I_{m,n}=[m2^{-n}, (m+1) 2^{-n}]$ and $\Delta_{m,n}=\max\limits_{t\in I_{m,n}}d_{\HH}(\cB(m 2^{-n}), \cB(t)) $ for $m, n\ge 0$. If $2^{-n/2}<u$ then by Lemma \ref{lem:hypBMtau} we have
\begin{align*}
P(\Delta_{m,n}\ge 2^{-n/2} u)&=P(\max_{0\le t\le 2^{-n}} d_{\HH}(\cB(0), \cB(t))\ge 2^{-n/2} u)
\le \frac{16}{ u\sqrt{\pi} }e^{-\frac{u^2}{16}}.
\end{align*}
Thus for $n\ge0$, $m\ge 0$ we get
\begin{align*}
P(\Delta_{m,n}\ge \tfrac{9}{2} \cdot2^{-n/2}\sqrt{\log\left(2^n+m+1\right)})\le \frac{3  \left(2^n+m+1\right)^{-5/4}}{   \sqrt{ \log\left(2^n+m+1\right)}}.
\end{align*}
We have $\sum_{n=0}^\infty \sum_{m=0}^\infty   \left(2^n+m+1\right)^{-5/4}<\infty$. By the Borel-Cantelli  lemma there is a random $N_0\ge 1$ so that if $n\ge N_0$  then
\begin{align}\label{aaa}
\Delta_{m,n}\le \tfrac92 \sqrt{2^{-n} \log\left(2^n+m+1\right)}.
\end{align}
We will show that (\ref{modcon}) holds with $0<h\le h_0=2^{-N_0}$ and $0\le s$. Let $m, n$ be  nonnegative integers with $2^{-n-1}<h\le 2^{-n}$ and $m 2^{-n}\le s\le (m+1)2^{-n}$. Then we have $n\ge N_0$, and
using the triangle inequality and (\ref{aaa}) we get
\begin{align*}
d_{\HH}(\cB(s),\cB(s+h))&\le 2 \Delta_{m,n}+\Delta_{m+1,n}\le 3\cdot \tfrac92  \sqrt{2^{-n} \log\left(2^n+m+2)\right)}\\
&\le 20 \sqrt{h \log(2+\tfrac{s+1}{h})}
\end{align*}
which finishes the proof.
\end{proof}

\noindent \textbf{Acknowledgements.}  The first author was partially supported by the NSF CAREER award DMS-1053280, the NSF award DMS-1712551 and the Simons Foundation. The second author was supported by the Canada Research Chair program the NSERC Discovery Accelerator grant, the 
MTA Momentum Random Spectra research group, and the ERC consolidated grant 648017 (Ab\'ert).


\bigskip\bigskip\bigskip\noindent
Benedek Valk\'o
\\Department of Mathematics
\\University of Wisconsin - Madison
\\Madison, WI 53706, USA
\\{\tt valko@math.wisc.edu}
\\[20pt]
B\'alint Vir\'ag
\\Departments of Mathematics and Statistics
\\University of Toronto
\\Toronto ON~~M5S 2E4, Canada
\\{\tt balint@math.toronto.edu}\\[5pt]
and\\[5pt]
Alfr\'ed R\'enyi Institute of Mathematics
\\
Hungarian Academy of Sciences
\\
Re\'altanoda utca 13-15,\\
H-1053, Budapest, Hungary

\end{document}